\documentclass[11pt]{article}
\usepackage[english]{babel}
\usepackage{amsmath,amsfonts,amssymb,amsthm,mathrsfs,bbm}
\usepackage[font=sf, labelfont={sf,bf}, margin=1cm]{caption}
\usepackage{graphicx,graphics}
\usepackage{epsfig}
\usepackage{latexsym}
\usepackage[applemac]{inputenc}
\usepackage{ae,aecompl}
\usepackage{pstricks}
\usepackage{enumerate}
\usepackage{xcolor}
\usepackage[pdfpagemode=UseNone,bookmarksopen=false,colorlinks=true,urlcolor=blue,citecolor=blue,citebordercolor=blue,linkcolor=blue]{hyperref}

\usepackage[nottoc,numbib]{tocbibind}

\usepackage[normalem]{ulem}
\usepackage[top=2.5cm,left=2.5cm,right=2.5cm,bottom=2.5cm]{geometry}
\usepackage{comment}

\usepackage{microtype,stmaryrd}          
\hypersetup{pdfstartview=XYZ}
\usepackage{array}
\usepackage[ruled]{algorithm}
\usepackage{algorithmic}
\usepackage{caption}
\usepackage{subcaption}
\usepackage[all, knot, cmtip]{xy} 
\usepackage{wrapfig}

\numberwithin{equation}{section}

\newcommand{\ndN}{\mathbb{N}}

\newcommand{\ndQ}{\mathbb{Q}}
\newcommand{\ndR}{\mathbb{R}}
\newcommand{\ndC}{\mathbb{C}}

\newcommand{\ndK}{\mathbb{K}}

\renewcommand{\Pr}[1]{\mathbb{P}\left(#1\right)}
\newcommand{\Ex}[1]{\mathbb{E}\left[#1\right]}

\newcommand{\one}{\mathbbm{1}}

\newcommand{\cA}{\mathcal{A}}
\newcommand{\cB}{\mathcal{B}}
\newcommand{\cC}{\mathcal{C}}

\newcommand{\cE}{\mathcal{E}}
\newcommand{\cF}{\mathcal{F}}
\newcommand{\cG}{\mathcal{G}}
\newcommand{\cH}{\mathcal{H}}

\newcommand{\cS}{\mathcal{S}}
\newcommand{\cT}{\mathcal{T}}

\newcommand{\cV}{\mathcal{V}}

\newcommand{\cX}{\mathcal{X}}

\newcommand{\mA}{\mathsf{A}}
\newcommand{\mB}{\mathsf{B}}
\newcommand{\mC}{\mathsf{C}}

\newcommand{\mK}{\mathsf{K}}

\newcommand{\mS}{\mathsf{S}}
\newcommand{\mT}{\mathsf{T}}

\newcommand{\mV}{\mathsf{V}}

\newcommand{\mX}{\mathsf{X}}

\newcommand{\me}{\mathsf{e}}

\newcommand{\CRT}{\mathcal{T}_\me}

\newcommand{\tF}{\tilde{\cF}}
\newcommand{\tG}{\tilde{\cG}}

\newcommand{\Sym}{\text{Sym}}
\newcommand{\RSym}{\text{RSym}}

\newcommand{\He}{\textsc{H}}
\newcommand{\Di}{\textsc{D}}

\newcommand{\he}{\text{h}}

\newcommand{\eqdist}{\,{\buildrel (d) \over =}\,}

\newcommand{\convdis}{\,{\buildrel d \over \longrightarrow}\,}
\newcommand{\convp}{\,{\buildrel p \over \longrightarrow}\,}

\newcommand \isom{
{\xrightarrow{\,\smash{\raisebox{-0.65ex}{\ensuremath{\scriptstyle\sim}}}\,}}
}

\newcommand{\disto}{\mathsf{dis}}

\newcommand{\Set}{\textsc{SET}}

\newtheorem{theorem}{Theorem}[section]

\newtheorem{corollary}[theorem]{Corollary}
\newtheorem{proposition}[theorem]{Proposition}
\newtheorem{lemma}[theorem]{Lemma}

\newtheorem{definition}[theorem]{Definition}
\newtheorem{remark}[theorem]{Remark}

\numberwithin{equation}{section}

\title{The continuum random tree is the scaling limit of unlabelled unrooted trees}

\date{}
\author{Benedikt Stufler}

\author{Benedikt Stufler\thanks{\'Ecole Normale Sup\'erieure de Lyon, E-mail: benedikt.stufler@ens-lyon.fr; The author is supported by the German Research Foundation DFG, STU 679/1-1}}

\begin{document}

	\maketitle

	\let\thefootnote\relax\footnotetext{ \\\emph{MSC2010 subject classifications}. Primary 60C05; secondary 05C80. \\
		\emph{Keywords and phrases.} unlabelled unrooted trees, continuum random tree, scaling limits}
	
	\vspace {-0.5cm}

\begin{abstract}
	We show that the uniform unlabelled unrooted tree with n vertices and vertex degrees in a fixed set converges in the Gromov--Hausdorff sense after a suitable rescaling to the Brownian continuum random tree. We also establish Benjamini--Schramm convergence of this model of random trees and provide a general approximation result, that allows for a transfer of a wide range of asymptotic properties of extremal and additive graph parameters from  P\'olya trees to unrooted trees.
\end{abstract}
\maketitle

\section{Introduction and main results}



Combinatorial trees are classical mathematical objects and crop up in variety of fields \cite{MR0025715,MR2484382,MR2718280}. In the present work we take a probabilistic approach to study unordered trees without labels. Here one distinguishes between P\'olya trees, which have a root, and unlabelled (unrooted) trees. It has been a long-standing conjecture by Aldous \cite[p. 55]{MR1166406}  that the continuum random tree (CRT) arises as scaling limit of these models of random trees. Marckert and Miermont \cite{MR2829313} treated the case of binary unordered rooted trees. The convergence of random (unrestricted) P\'olya trees was confirmed by Haas and Miermont \cite{MR3050512} using new methods, and an alternative proof has been given later by Panagiotou and Stufler \cite{2015arXiv150207180P}. As was also mentioned  in~\cite{MR3050512}, this does not settle the question regarding the convergence of random unlabelled unrooted trees. The main challenge for these structures the  complexity of their symmetries. Rooted trees have a simpler structure, as any automorphism is required to fix the root vertex. Our first main result confirms the CRT as scaling limit of unlabelled unrooted trees as their number of vertices becomes large, confirming Aldous conjecture for these structures. We take a unified approach to cover all (sensible) cases of vertex degree restrictions.

\begin{figure}[t]
	\centering
	\begin{minipage}{1\textwidth}
		\centering
		\includegraphics[width=0.30\textwidth]{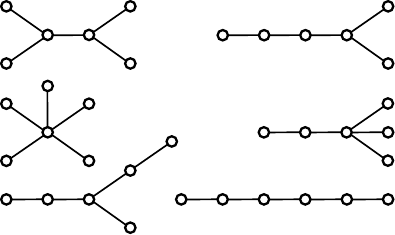}
		\caption{All unlabelled unrooted trees with $6$ vertices.}
		\label{fi:treelist}
	\end{minipage}
\end{figure}

Throughout, we let $\Omega$ denote a fixed set of positive integers containing $1$ and at least one integer equal or larger than $3$, and set $\Omega^* = \Omega - 1$. Let $\mT_n$ be drawn uniformly at random from the unlabelled trees with $n$ vertices and vertex-degrees in $\Omega$, and let $\mA_{n-1}$ denote the random P\'olya tree selected uniformly among all such trees with $n-1$ vertices and outdegrees in the shifted set $\Omega^*$. See Figure~\ref{fi:treelist} and \ref{fi:polya} for illustrations of these structures.

\begin{theorem}
	\label{SeFrte:main2}
	There is a constant $e_\Omega$ such that
	\begin{align}
		\label{eq:convunrooted}
		(\mT_n, e_\Omega n^{-1/2} d_{\mT_n})  \convdis (\CRT, d_{\CRT})
	\end{align}
	in the Gromov--Hausdorff sense, as $n \equiv 2 \mod \gcd(\Omega^*)$ becomes large. Moreover, there are constants $C,c>0$ such that the diameter $\Di(\mT_n)$ satisfies the tail bound
	\begin{align}
		\label{eq:tailbound}
		\Pr{\Di(\mT_n) \ge x} \le C \exp(-c x^2/n)
	\end{align}
	for all $n$ and $x \ge 0$.
\end{theorem}

The CRT plays a central role in the study of the geometric shape of large discrete structures. It crops up as scaling limit for a variety of models \cite{MR2438817,MR3335010,Ca,MR3382675,MR3342658,2014arXiv1411.1865P} and incited research in further directions \cite{2015arXiv150906616A,MR3399812}. Although scaling limits describe asymptotic global properties, they do not contain information on local properties, such as the limiting degree distribution of a randomly chosen vertex in a graph. Such asymptotic local properties of random rooted structures are described by Benjamini--Schramm limits \cite{MR2013797,MR3010812,MR1873300}.  Our second main result establishes Benjamini--Schramm convergence for random unlabelled unrooted trees toward an infinite limit tree. We take a unified approach to cover all sensible cases of vertex degree restrictions.

\begin{theorem}
	\label{te:locconv}
	The random unrooted tree $\mT_n$ converges in the Benjamini--Schramm sense toward an infinite rooted tree ${\mA_{\Omega^*}}$, as $n \equiv 2 \mod \gcd(\Omega^*)$ becomes large. Even stronger, if $v_n$ denotes a uniformly at random selected vertex of the tree $\mT_n$, then for each sequence $k_n = o(\sqrt{n})$ the radius $k_n$ graph neighbourhood $V_{k_n}(\cdot)$ satisfies
	\begin{align}
		\label{eq:fail}
		d_{\textsc{TV}}(V_{k_n}(\mT_n, v_n), V_{k_n}(\mA_{\Omega^*})) \to 0.
	\end{align}
\end{theorem}

Here $d_{\textsc{TV}}$ denotes the total variation distance. Note that this form of convergence is best possible, as \eqref{eq:fail} fails if the order of $k_n$ is comparable to $\sqrt{n}$. In the case $\Omega = \ndN$, Benjamini--Schramm convergence for $\mT_n$ was independently obtained by Georgakopoulos and Wagner  \cite{2015arXiv151203572G} using different techniques. 
Our methods for the proof of Theorems \ref{SeFrte:main2} and \ref{te:locconv} are based on the cycle pointing decomposition established recently by Bodirsky, Fusy, Kang and Vigerske \cite{MR2810913}. This novel and effective centering method differs fundamentally from classical approaches, such as the geometric center, and applies to arbitrary classes of combinatorial structures.  We use it to approximate the random unlabelled unrooted tree $\mT_n$ with $n$ vertices and vertex outdegrees in a set $\Omega$, by random P\'olya trees with vertex outdegrees in the shifted set $\Omega^* = \Omega -1$, whose random sizes concentrate around $n$. The approximation works not only for graph limits, but actually for a large range of additive and extremal graph parameters.

\begin{figure}[t]
	\centering
	\begin{minipage}{1\textwidth}
		\centering
		\includegraphics[width=0.4\textwidth]{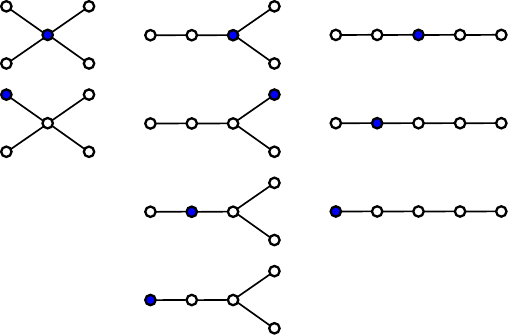}
		\caption{All P\'olya trees with $5$ vertices.}
		\label{fi:polya}
	\end{minipage}
\end{figure}

\begin{theorem}
	\label{te:appr}
 There are constants $C,c>0$, a random number $K_n = n + O_p(1) \le n$, and a coupling of the randomly sized P\'olya tree $\mA_{K_n}$ with a tree $\mB_n$  having stochastically bounded size $n - K_n+1$, such that the random tree $\bar{\mT}_n$ obtained by identifying the root vertices of $\mA_{K_n}$ and $\mB_n$ satisfies 
	\[
		d_{\textsc{TV}}(\mT_n, \bar{\mT}_n) \le C \exp(-cn)
	\]
	for all $n$.
\end{theorem}


Theorem~\ref{te:appr} establishes in full generality how a random unrooted tree may be approximated by a single large random rooted tree having  the property, that when conditioned on having a fixed size, it is uniformly distributed among all P\'olya trees with this size and the given vertex outdegree restrictions.  This has far reaching consequences and underline the advantages of this approach. It implies that for a very large set of graph theoretic properties (maximum degree, degree distribution, subtree counts, \ldots) everything known (present and future) about  random P\'olya trees also applies to random unlabelled unrooted trees, erasing the need to study uniform unrooted unordered trees directly. For example, Haas and Miermont~\cite[Thm. 9, Cor. 10]{MR3050512} established Gromov--Hausdorff--Prokhorov scaling limits for uniform unordered rooted trees endowed with the uniform measure on their leaves or on all their vertices, if the vertex out-degrees are restricted to a set of the form $\Omega^* = \ndN_0$, $\Omega^*=\{0,d\}$ or $\Omega^*=\{0, \ldots, d\}$ for some $d \ge 2$. Using this result, it follows easily from Theorem~\ref{te:appr} that  the uniform vertex degree restricted unrooted tree $\mT_n$ with vertex degrees in $\Omega = \Omega^* +1$ also converges in the Gromov--Hausdorff--Prokhorov sense, thus strengthening the convergence of Theorem~\ref{SeFrte:main2} for these cases. But again, it is not about for which cases of vertex-degree restrictions we may deduce convergence at the moment. The contribution of Theorem~\ref{te:appr} is that "practically all" properties of random unordered rooted trees get transferred automatically to the unrooted case, regardless of the extend to which they are understood at present. 

Thus, Theorem~\ref{te:appr} provides a rigorous justification of the empirically backed and widely believed fact that rooted and unrooted trees behave asymptotically similarly. Note that this does \emph{not} imply that almost all unrooted trees are asymmetric (meaning the absence of non-trivial symmetries) or possess as much possible root locations as vertices. Some discrete structures such as planar maps with half-edges as atoms have such properties, and hence a purely enumerative argument suffices to show that the asymptotic study of these objects  is equivalent to the study of half-edge rooted planar maps. The case of unordered trees is different, as the probability for the random tree $\mT_n$ to be asymmetric is bounded away from $1$, as is the probability for the event that rooting it at each of its $n$ vertices yields $n$ distinct trees. Moreover, the approximation argument of  Theorem~\ref{te:appr}  does not appear to work as well in the other direction. For example, the convergence of $\mT_n$ (in the local sense, or in the sense of scaling limits) may be used to obtain convergence of a random P\'olya-tree having a {\em random} number of vertices (depending on $n$), but, although this number concentrates, this is not sufficient to deduce convergence of a random P\'olya tree with a deterministic size  that becomes large. Hence the most economic approach is really to study P\'olya trees and then transfer the results to random unlabelled unrooted trees. Furthermore, in  \cite{2015arXiv150207180P} it was shown how asymptotic properties of conditioned Galton--Watson trees may be transferred to random P\'olya trees, which by the results of the present work hence also apply to the unrooted model. As Galton--Watson trees are without doubt the best understood model of random trees in probability theory, it is natural to pave the way for building on this knowledge.

In \cite{MR2810913} the cycle pointing method was developed for the enumeration and efficient sampling of discrete structures. The present work demonstrates for the important classical example of unlabelled  trees how a combination with a probabilistic approach allows us to answer a large number of questions related to the study of asymptotic properties of random discrete structures. Due to the generality of the involved methods this will likely stimulate probabilistic applications to further classes of discrete structures, such as models of random unlabelled graphs. 

\subsection{Combinatorial applications of the scaling limit}
A direct consequence of the scaling limit in Theorem~\ref{SeFrte:main2} is that the rescaled diameter $e_\Omega n^{-1/2} \Di(\mT_n)$ converges weakly and in arbitrarily high moments toward the diameter $\Di(\CRT)$ of the CRT. That is,  \[
\Pr{n^{-1/2} e_\Omega \Di(\mT_n) >  x} \to \Pr{\Di(\CRT) > x},
\]
and
\[
\Ex{\Di(\mT_n)^p } \sim e_\Omega^{-p} n^{p/2} \Ex{\Di(\CRT)^p}.
\]
The distribution of $\Di(\CRT)$ is known and given by
\begin{align}
\label{eq:diamcrt0}
\Di(\CRT) \eqdist \sup_{0 \le t_1 \le t_2 \le 1}(\me(t_1) + \me(t_2) - 2 \inf_{t_1 \le t \le t_2} \me(t)),
\end{align}
with $\me = (\me_t)_{0 \le t \le 1}$ denoting Brownian excursion of length $1$, and
\begin{align}
\label{eq:diamcrt}
\Pr{\Di(\CRT) > x} = \sum_{k=1}^\infty (k^2-1)\Big(\frac{2}{3}k^4x^4 -4k^2x^2 +2\Big)\exp(-k^2x^2/2).
\end{align}
Equations \eqref{eq:diamcrt0} and \eqref{eq:diamcrt}  have been established by  Aldous \cite[Ch. 3.4]{MR1166406} using convergence of random discrete trees. Expression~\eqref{eq:diamcrt} was recently recovered directly in the continuous setting by Wang~\cite{MR3434205}.  The moments of the diameter are given by:
\begin{align}
\label{eq:momdi1}
\Ex{\Di(\CRT)} &= \frac{4}{3}\sqrt{\pi/2}, \quad \Ex{\Di(\CRT)^2} = \frac{2}{3}\left(1 + \frac{\pi^2}{3}\right), \quad \Ex{\Di(\CRT)^3} = 2 \sqrt{2\pi}, \\
\label{eq:momdi2}
\Ex{\Di(\CRT)^k} &= \frac{2^{k/2}}{3} k(k-1)(k-3) \Gamma(k/2)(\zeta(k-2) - \zeta(k)) \quad \text{for $k \ge 4$}.
\end{align}
The expression $\Ex{\Di(\CRT)} = \frac{4}{3}\sqrt{\pi/2}$ may be obtained as shown in Aldous \cite[Sec. 3.4]{MR1166406} using results of Szekeres \cite{MR731595}, who proved the existence of a limit distribution for the diameter of rescaled random unordered labelled trees. The higher moments could be obtained in the same way by elaborated calculations, or, we can deduce them by combining Theorem~\ref{SeFrte:main2} with results by Broutin and Flajolet, who studied in \cite{MR2956055} the random tree $\tau_n$ that is drawn uniformly at random among all unlabelled trees with $n$ leaves in which each inner vertex is required to have degree $3$. Using analytic methods \cite[Thm. 8]{MR2956055}, they computed asymptotics of the form \[\Ex{\Di(\tau_n)^r} \sim c_r \lambda^{-r} n^{r/2}\] with $\lambda$ an analytically given constant, and the constants $c_r$ given by
\begin{align*}
c_1 &= \frac{8}{3} \sqrt{\pi}, \quad c_2 = \frac{16}{3}(1 + \frac{\pi^2}{3}), \quad c_3 = 64 \sqrt{\pi}, \\ 
c_r &= \frac{4^r}{3}r(r-1)(r-3)\Gamma(r/2)(\zeta(r-2) - \zeta(r)) \quad \text{if $r \ge 4$}.
\end{align*} 
As $\tau_n$ has $n$ leaves and hence $2n-1$ vertices in total, it follows by Theorem~\ref{SeFrte:main2} that
\[
(\tau_n, e_{\{0,2\}} (2n-1)^{-1/2} d_{\tau_n}) \convdis (\CRT, d_{\CRT})
\]
and consequently, by the exponential tail-bounds for the diameter in Theorem~\ref{SeFrte:main2}, which imply arbitrarily high uniform integrability,
\[
\Ex{\Di(\tau_n)^r} \sim \Ex{\Di(\CRT)^r} (e_{\{0,2\}} / \sqrt{2})^{-r}  n^{r/2}.
\] It follows that
\[
\Ex{\Di(\CRT)^r} = c_r (e_{\{0,2\}} / (\sqrt{2} \lambda))^{r}.\] All that remains is to calculate the ratio $e_{\{0,2\}} / (\sqrt{2} \lambda)$, which is given by \[e_{\{0,2\}} / (\sqrt{2} \lambda) = \Ex{\Di(\CRT)} / c_1 = 1/(2 \sqrt{2}),\] since $\Ex{\Di(\CRT)} = 4/3 \sqrt{\pi/2}$. This elegantly  yields Equations~\eqref{eq:momdi1} and \eqref{eq:momdi2}.

\subsection*{Outline of the paper}
In Section~\ref{sec:discretetrees} we fix basic notions on graphs and discrete trees. Section~\ref{sec:crt} gives a brief account on Gromov--Hausdorff convergence and the continuum random tree. Section~\ref{sec:local} recalls the notion of local weak convergence and results for random P\'olya trees. Section~\ref{sec:species} introduces the reader to the language of combinatorial species, and Section~\ref{sec:cycle} to the technique of cycle pointing that is formulated using these notions. Section~\ref{sec:Boltzmann} recalls the concept of (P\'olya-)Boltzmann samplers, which builds a bridge from  combinatorial structures to random algorithms that sample these structures. Section~\ref{sec:mult} recalls a result related to extremal component sizes in random multisets.  In Section \ref{sec:mainproof} we present the proofs of our main results.

\subsection*{Notation}
Throughout, we set
\[
\ndN=\{1,2,\ldots\}, \qquad \ndN_0 = \{0\} \cup \ndN, \qquad [n]=\{1,2,\ldots, n\}, \qquad n \in \ndN_0.
\]
we assume that all considered random variables are defined on a common probability space whose measure we denote by $\mathbb{P}$. All unspecified limits are taken as $n$ becomes large, possibly along a shifted sublattice of the integers.  We write $\convdis$ and $\convp$ for convergence in distribution and probability, and $\eqdist$ for equality in distribution. An event holds with high probability, if its probability tends to $1$ as $n \to \infty$.
We let $O_p(1)$ denote an unspecified random variable $X_n$ of a stochastically bounded sequence $(X_n)_n$. The total variation distance of measures and random variables is denoted by $d_{\textsc{TV}}$. For a sequence $a_n$ that is eventually positive the notation $O(a_n)$ and $o(a_n)$ refer to unspecified deterministic sequences that are bounded by a multiple of $a_n$ or whose order is negligible compared to $a_n$. Given a multi-variate power series $f(z_1, z_2, \ldots)$ we let $[z_1^{t_1} \cdots z_m^{t_m}]f(z_1, z_2, \ldots)$ denote the coefficient corresponding to the monomial $z_1^{t_1} \cdots z_m^{t_m}$.

\section{Discrete trees}
\label{sec:discretetrees}

\label{sec:cp1}
A {\em (labelled) graph} $G$ consists of a non-empty set $V(G)$ of {\em vertices} (or {\em labels}) and a set $E(G)$ of {\em edges} that are two-element subsets of $V(G)$. The cardinality $|V(G)|$ of the vertex set is termed the  {\em size} of $G$. Instead of $v \in V(G)$ we will often just write $v \in G$. Two vertices $v, w \in V(G)$ are said to be {\em adjacent} if $\{v,w\} \in E(G)$. An edge $e \in E(G)$ is adjacent to $v$ if $v \in e$. The cardinality of the set of all edges adjacent to a vertex $v$ is termed its {\em degree} and denoted by $d_G(v)$.  We say the graph $G$ is \emph{connected} if any two vertices $u, v \in V(G)$ are connected by a path in $G$. The length of a shortest path connecting the vertices $u$ and $v$ is called the {\em graph distance} of $u$ and $v$ and it is denoted by $d_G(u,v)$. Clearly $d_G$ is a metric on the vertex set $V(G)$. A graph~$G$ together with a distinguished vertex $v \in V(G)$ is called a {\em rooted} graph with root-vertex $v$. The {\em height} $\he(w)$ of a vertex $w \in V(G)$ is its distance from the root. The \emph{height} $\He(G)$ of the entire graph is the supremum of the heights  of the vertices in $G$. Two graphs $G_1$ and $G_2$ are termed {\em isomorphic}, if there is a bijection $\varphi: V(G_1) \to V(G_2)$ such that any two vertices $x,y \in V(G_1)$ are adjacent in $G_1$ if and only if $\phi(x)$ and $\phi(y)$ are adjacent in $G_2$. Any such bijection is termed an {\em isomorphism} between $G_1$ and $G_2$. Rooted graphs $G_1^\bullet = (G_1, o_1)$ and $G_2^\bullet=(G_2, o_2)$ are termed isomorphic, if there is a graph isomorphism $\phi$ from $G_1$ to $G_2$ that satisfies $\phi(o_1) = o_2$. An isomorphism class of (rooted) graphs is also called an {\em unlabelled (rooted) graph}. We will often not distinguish between such a class or any fixed representative of that class.

A {\em tree} $T$ is a non-empty connected graph without cyclic subgraphs, that is, we cannot walk from one vertex to itself without crossing at least one edge twice. Any two vertices of a tree are connected by a unique path. Figure~\ref{fi:treelist} depicts the list of all unlabelled trees with $6$ vertices.  If $T$ is rooted, then the vertices $w' \in V(T)$ that are adjacent to a vertex $w$ and have height $\he(w') = \he(w) + 1$ form the {\em offspring set} of the vertex $w$. Its cardinality is the {\em outdegree} $d^+_T(w)$ of the vertex $w$. Unlabelled rooted trees are also termed P\'olya trees.  Note that while any labelled tree with $n$ vertices admits $n$ different roots, this does not hold in the unlabelled setting. For example, as illustrated in Figure~\ref{fi:polya}, there are $3$ unlabelled trees with  $5$ vertices and each of them has a different number of rootings.

\section{Scaling limits}
\label{sec:crt}
We briefly recall several relevant results regarding the convergence of random rooted trees toward the continuum random tree.

\subsection{Gromov--Hausdorff convergence}
We introduce the required notions regarding the Gromov--Hausdorff convergence following Burago, Burago and Ivanov \cite[Ch. 7]{MR1835418} and Le Gall and Miermont \cite{MR3025391}

\subsubsection{The Hausdorff metric}
\label{sec:hausdorff}
Recall that given subsets $A$ and $B$ of a metric space $(X,d)$, their  {\em Hausdorff-distance} is given by
\[
d_{\text{H}}(A,B) = \inf\{\epsilon > 0 \mid A \subset U_{\epsilon}(B), B \subset U_{\epsilon}(A)\} \in [0, \infty],
\]
where $U_{\epsilon}(A) = \{x \in X \mid d(x,A) \le \epsilon\}$ denotes the {\em $\epsilon$-hull} of $A$. In general, the Hausdorff-distance does not define a metric on the set of all subsets of $X$, but it does on the set of all compact subsets of $X$ (\cite[Prop. 7.3.3]{MR1835418}).

\subsubsection{The Gromov--Hausdorff distance}

The Gromov--Hausdorff distance allows us to compare arbitrary metric spaces, instead of only subsets of a common metric space. It is defined by the infimum of Hausdorff-distances of isometric copies in a common metric space. We are also going to consider a variation of the Gromov--Hausdorff distance given in \cite{MR3025391} for {\em pointed} metric spaces, which are metric spaces together with a distinguished point.

Given metric spaces $(X,d_X)$, and $(Y,d_Y)$, and distinguished elements $x_0 \in X$ and $y_0 \in Y$, the Gromov--Hausdorff distances of $X$ and $Y$ and the pointed spaces $X^\bullet = (X,x_0)$ and $Y^\bullet = (Y,y_0)$ are defined by
\begin{align*}
d_{\text{GH}}(X,Y) &= \inf_{\iota_X, \iota_Y} d_{\text{H}}(\iota_X(X), \iota_Y(Y)) \in [0, \infty], \\
d_{\text{GH}}(X^\bullet,Y^\bullet) &= \inf_{\iota_X, \iota_Y} \max\left\{d_{\text{H}}(\iota_X(X), \iota_Y(Y)), d_E(\iota_X(x_0), \iota_Y(y_0))\right\} \in [0, \infty]
\end{align*}
where in both cases the infimum is taken over all isometric embeddings $\iota_X: X \to E$ and $\iota_Y: Y \to E$ into a common metric space $(E, d_E)$, compare with Figure~\ref{fi:gromov}.

\begin{figure}[ht]
	\centering
	\begin{minipage}{1.0\textwidth}
		\centering
		\includegraphics[width=0.4\textwidth]{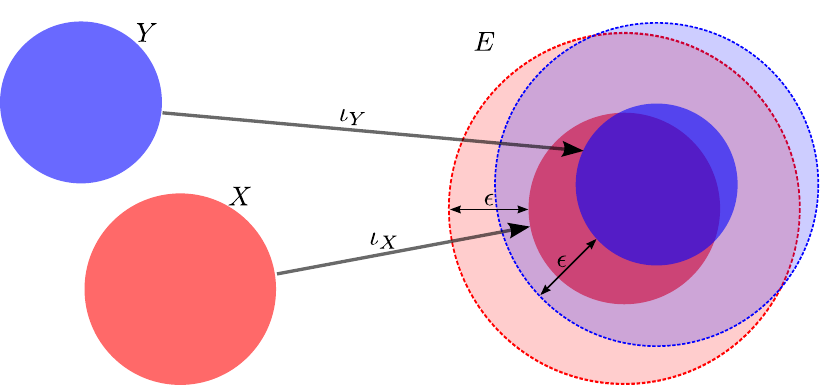}
		\caption{The Gromov--Hausdorff distance.}
		\label{fi:gromov}
	\end{minipage}
\end{figure}

We will make use of the following characterisation of the Gromov--Hausdorff metric. Given two metric spaces $(X, d_X)$ and $(Y, d_Y)$ a \emph{correspondence} between them is a relation $R \subset X \times Y$ such that any point $x \in X$ corresponds to at least one point $y \in Y$ and vice versa. If $X$ and $Y$ are pointed, we additionally require that the roots correspond to each other. The \emph{distortion} of $R$ is given by
\[
\disto(R) = \sup\{|d_X(x_1,x_2) - d_Y(y_1,y_2)| \mid (x_1, y_1),(x_2,y_2) \in R \}.
\]
\begin{proposition}[{\cite[Thm. 7.3.25]{MR1835418} and \cite[Prop.\ 3.6]{MR3025391}}]
	\label{pro:distortion}
	Given two metric spaces $X, Y$ and pointed metric spaces $X^\bullet, Y^\bullet$ we have that
	\begin{align*}
	d_{\text{GH}}(X,Y) = \frac{1}{2}\inf_R \disto(R), \quad \text{and} \quad
	d_{\text{GH}}(X^\bullet,Y^\bullet) = \frac{1}{2} \inf_R \disto(R),
	\end{align*}
	where $R$ ranges over all correspondences between $X$ and $Y$ (or $X^\bullet$ and $Y^\bullet$).
\end{proposition}

Using this reformulation of the Gromov--Hausdorff distance, one may check that it satisfies the following properties.

\begin{lemma}[{\cite[Thm. 7.3.30]{MR1835418} and \cite[Thm.\ 3.5]{MR3025391}}]
	\label{le:axioms}
	Let $X$, $Y$, and $Z$ be (pointed) metric spaces. Then the following assertions hold. 
	\begin{enumerate}[\quad i)]
		\item $d_{\text{GH}}(X,Y) = 0$ if and only if $X$ and $Y$ are isometric.
		\item $d_{\text{GH}}(X,Z) \le d_{\text{GH}}(X,Y) + d_{\text{GH}}(Y,Z)$.
		\item If $X$ and $Y$ are bounded, then $d_{\text{GH}}(X,Y) < \infty$.
	\end{enumerate}
\end{lemma}

\subsubsection{The space of isometry classes of compact metric spaces}

In Section~\ref{sec:hausdorff} we saw that the Hausdorff-distance defines a metric on the set of all compact subsets of a metric space. By Lemma~\ref{le:axioms} the Gromov--Hausdorff distance satisfies in a similar way the axioms of a (finite) pseudo-metric on the class of all compact metric spaces, and two metric spaces have Gromov--Hausdorff distance $0$ if and only if they are isometric. Informally speaking, this yields a metric on the collection of all isometry classes of metric spaces, and in a similar way we may endow the collection of isometry classes of pointed metric spaces with a metric.

Note that from a formal viewpoint this construction is a bit problematic, since we are forming a collection of proper classes (as opposed to sets). A solution is presented as an exercise in {\cite[Rem. 7.2.5]{MR1835418}}: 

\begin{proposition}
	Any set of pairwise non-isometric (pointed) metric spaces has cardinality at most $2^{\aleph_0}$, and there are specific examples of  $2^{\aleph_0}$ many non-isometric (pointed) spaces.
\end{proposition}

We may thus fix a representative of each isometry class of (pointed) metric spaces and let $\ndK$ (resp. $\ndK^\bullet$) denote the resulting sets of spaces. Lemma~\ref{le:axioms} now reads as follows.

\begin{corollary}[{\cite[Thm. 7.3.30]{MR1835418}}]
	The Gromov--Hausdorff distance defines a finite metric on the set $\ndK$ (resp. $\ndK^\bullet$) of representatives of isometry classes of (pointed) compact metric spaces.
\end{corollary}

The metric spaces $\ndK$ and $\ndK^\bullet$ have nice properties, which make them  suitable for studying random elements:

\begin{proposition}[{\cite[Thm.\ 3.5]{MR3025391} and \cite[Thm. 7.4.15]{MR1835418}}]
	The spaces $\ndK$ and $\ndK^\bullet$ are separable and complete, i.e. they are Polish spaces.
\end{proposition}

\subsection{The continuum random tree}
An {\em $\ndR$-tree} is a metric space $(X,d)$ such that for any two points $x,y \in X$ the following properties hold 
\begin{enumerate}[\qquad 1.]
	\item There is a unique isometric map from the interval $\varphi_{x,y}: [0, d_f(x,y)] \to X$ satisfying $\varphi_{x,y}(0) = x$ and $\varphi_{x,y}(d_f(x,y)) = y$.
	\item If $q: [0, d_f(x,y)] \to X$ is a continuous injective map, then 
	\[
	q([0, d_f(x,y)]) = \varphi_{x,y}([0, d_f(x,y)]).
	\]
\end{enumerate}

$\ndR$-trees may be constructed as follows. Let $f: [0,1] \to [0, \infty[$ be a continuous function satisfying $f(0) = f(1) = 0$. Consider the pseudo-metric $d$ on the interval $[0,1]$ given by
\[
d(u,v) = f(u) + f(v) - 2 \inf_{u \le s \le v} f(s)
\]
for $u \le v$. Let $(\cT_f, d_{\cT_f)} = ([0,1]/ \mathord{\sim}, \bar{d})$ denote the corresponding quotient space. We may consider this space as rooted at the equivalence class $\bar{0}$ of $0$.

\begin{proposition}[{\cite[Thm.\ 3.1]{MR3025391}}]
	Given a continuous function $f: [0,1] \to [0, \infty[$ satisfying $f(0)=f(1)$ the corresponding metric space $\cT_f$ is a compact $\ndR$-tree.
\end{proposition}

Hence, this construction defines a map from a set of continuous functions to the space $\ndK^\bullet$. It can be seen to be Lipschitz-continuous:

\begin{proposition}[{\cite[Cor.\ 3.7]{MR3025391}}]
	The map 
	\[
	(\{f \in \cC([0,1], \ndR_{\ge 0}) \mid f(0)=f(1)=0 \}, \lVert \cdot \lVert_{\infty}) \to (\ndK^\bullet, d_{GH}), \qquad f \mapsto \cT_f
	\] is Lipschitz-continuous.
\end{proposition}

Hence we may define the continuum random tree as a random element of the polish space~$\ndK^\bullet$.

\begin{definition}
	The random pointed metric space $(\CRT, d_{\CRT}, \bar{0})$ coded by the Brownian excursion of duration one $\me=(\me_t)_{0 \le t \le 1}$ is called the Brownian continuum random tree (CRT).
\end{definition}

Note that the Lipschitz-continuity (and hence measurability) of the above map ensures that the CRT is a random variable.

\subsection{Scaling limits of random P\'olya trees}

It is  known that for any subset $\Omega^* \subset \ndN_0$ containing zero and at least one integer $k \ge 2$, the P\'olya tree $\mA_n$ drawn uniformly at random from the set of all P\'olya trees with $n$ vertices and vertex outdegrees in the set $\Omega^*$ admits the CRT as scaling limit. That is, there is a constant $c_{\Omega^*}$ satisfying
\begin{align}
\label{eq:convpolya}
(\mA_n, c_{\Omega^*} n^{-1/2} d_{\mA_n}) \convdis (\CRT, d_{\CRT})
\end{align}
as random elements of the space $\ndK^\bullet$. This has been shown by Marckert and Miermont for the case $\Omega^*=\{0,2\}$ in \cite{MR2829313}. Using different techniques that built on general results for Markov branching trees, Haas and Miermont \cite{MR3050512} extend this result to the cases $\Omega^*=\{0,d\}$ for all $d \ge 2$, and $\Omega^* = \ndN_0$. A unified approach for all sensible vertex outdegree restrictions (that is, requiring only $0 \in \Omega^*$ and $k \in \Omega^*$ for at least one $k \ge 2$) was taken in \cite{2015arXiv150207180P}, using combinatorial techniques and obtaining  tail bounds for the diameter $\Di(\mA_{n})$ of the form
\begin{align}
\label{SeFrle:tailrooted}
\Pr{\Di(\mA_n) \ge x} \le C \exp(-c x^2)
\end{align}
for all $x \ge 0$.

\section{Local weak limits}
\label{sec:local}

We briefly recall relevant notions and results regarding the local convergence of random rooted trees.

\subsection{The metric for local convergence}
Given two rooted, locally finite (that is, the graph may have infinitely many vertices, but each vertex has only finitely many neighbours) connected graphs $G^\bullet = (G, o_G)$ and $
H^\bullet = (H, o_H)$, we may consider the distance
\[
d_{\text{BS}}(G^\bullet, H^\bullet) =  2^{-\sup \{k \in \ndN_0 \,\mid\, V_k(G^\bullet) \simeq V_k(H^\bullet) \}}
\]
with $V_k(G^\bullet)$ denoting the subgraph of $G$ induced by all vertices with graph-distance at most $k$ from the root-vertex $o_G$. Here $V_k(G^\bullet) \simeq V_k(H^\bullet)$ denotes isomorphism of rooted graphs, that is, the existence of a graph isomorphism $\phi: V_k(G^\bullet) \to V_k(H^\bullet)$ satisfying $\phi(o_G) = o_H$. This defines a premetric on the collection of all rooted connected locally finite graphs. 

If $\mathbb{B}$ denotes the collection of isomorphism classes of rooted locally finite connected graphs ("unlabelled rooted graphs"), then the (lift of) this distance defines a metric on $\mathbb{B}$ which is complete and separable, i.e. $(\mathbb{B}, d_{\text{BS}})$ is a Polish space. Similarly as for the Gromov--Hausdorff metric, we may safely ignore the fact that $\mathbb{B}$ is a collection of proper classes (as opposed to sets). In order to precise, we would only need to fix a representatives of each isomorphism class and work with the set of these representatives instead.

\subsection{Benjamini--Schramm convergence of random P\'olya trees}

Let $\Omega \subset \ndN$ denote a subset containing $1$ and at least one integer $k \ge 3$, and let $\Omega^* = \Omega -1$ denote the shifted set. Let $\mA_n$ denote the random tree drawn uniformly at random from the set of all P\'olya trees with $n$ vertices and vertex outdegrees in $\Omega^*$. Let  $u_n$ denote a uniformly at random drawn selected vertex of $\mA_n$. It was shown in \cite[Thm. 6.22]{2015arXiv150402006S}, that there is a random infinite rooted trees $\mA_{\Omega^*}$ such that for each sequence $k_n = o(\sqrt{n})$ the random vertex $u_n$ has with high probability height strictly larger than $k_n$ in the tree $\mA_n$ and
\begin{align}
\label{eq:localconv}
d_{\textsc{TV}}( V_{k_n}(\mA_n, u_n), V_{k_n}(\mA_{\Omega^*})) \to 0.
\end{align}

\section{Combinatorial species of structures}
\label{sec:species}
Combinatorial species were developed by Joyal \cite{MR633783} and allow for a systematic study of a wide range of combinatorial objects. We are going to make heavy use of this framework and recall the required theory and notation following Bergeron, Labelle and Leroux \cite{MR1629341} and Joyal \cite{MR633783}. The language of {\em combinatorial classes} used in the book on analytic combinatorics by Flajolet and Sedgewick~\cite{MR2483235} is essentially equivalent in many aspects, although less emphasis is put on studying objects up to symmetry. 

\subsection{Combinatorial species of structures}
\label{sec:cp3}
A {\em combinatorial species} may be defined as a functor $\cF$ that maps any finite set $U$ of {\em labels} to a finite set $\cF[U]$ of {\em $\cF$-objects} and any bijection $\sigma: U \to V$ of finite sets to its (bijective) {\em transport function} $\cF[\sigma]:\cF[U]\to\cF[V]$ {\em along} $\sigma$, such that composition of maps and the identity maps are preserved. Formally, a species is a functor from the groupoid of finite sets and bijections to the category of finite sets and arbitrary maps. 
We say that a species $\cG$ is a {\em subspecies} of $\cF$, and write $\cG \subset \cF$, if $\cG[U] \subset \cF[U]$ for all finite sets $U$ and $\cG[\sigma] = \cF[\sigma]|_{U}$ for all bijections $\sigma: U \to V$.
Given two species $\cF$ and $\cG$, an {\em isomorphism} $\alpha: \cF \, \isom\,  \cG$ from $\cF$ to $\cG$ is a family of bijections $\alpha = (\alpha_U: \cF[U] \to \cG[U])_{U}$ where $U$ ranges over all finite sets, such that for all bijective maps $\sigma: U \to V$ the following diagram commutes. 
\[
\xymatrix{ \cF[U] \ar[d]^{\alpha_U} \ar[r]^{\cF[\sigma]} &\cF[V]\ar[d]^{\alpha_V}\\
	\cG[U] \ar[r]^{\cG[\sigma]} 		    &\cG[V]}
\]
In other words, $\alpha$ is a natural isomorphism between these functors. The species $\cF$ and $\cG$ are {\em isomorphic} if there exists and isomorphism from one to the other. This is denoted by $\cF \simeq \cG$. 

An element $F_U \in \cF[U]$ has size $|F_U| := |U|$ and two $\cF$-objects $F_U$ and $F_V$ are termed {\em isomorphic} if there is a bijection $\sigma: U \to V$ such that $\cF[\sigma](F_U) = F_V$. We will often just write $\sigma.F_U = F_V$ instead, if there is no risk of confusion.  We say $\sigma$ is an {\em isomorphism} from $F_U$ to $F_V$. If $U=V$ and $F_U = F_V$ then $\sigma$ is an {\em automorphism} of $F_U$. An isomorphism class of $\cF$-structures is called an \emph{unlabelled} $\cF$-object or an {\em isomorphism type}. By abuse of notation, we treat unlabelled objects as if they were regular objects. We will also just write $F \in \cF$ to state that $F$ is an $\cF$-object.

We will mostly be interested in the species of labelled trees. Moreover, we will make use of standard species such as the $\Set$-species given by $\Set[U] = \{U\}$ for all $U$. Moreover, we let $\cX$ the species with a single object of size $1$.
\subsection{Symmetries and generating power series}
Letting $\tilde{f}_n$ denote the number of unlabelled $\cF$-objects of size $n$, the {\em ordinary generating series} of $\cF$ is defined by
\[
\tilde{\cF}(x) = \sum_{n=0}^\infty \tilde{f}_n x^n
\]
A pair $(F, \sigma)$ of an $\cF$-object together with an automorphism is called a {\em symmetry}. Its {\em weight monomial} is given by \[
w_{(F, \sigma)} = \frac{1}{n!} x_1 ^{\sigma_1} x_2^{\sigma_2} \cdots x_n^{\sigma_n} \in \ndQ[[x_1, x_2, \ldots]]
\] with $n$ denoting the size of $F$ and $\sigma_i$ denoting the number of $i$-cycles of the permutation $\sigma$. In particular $\sigma_1$ denotes the number of fixpoints. We may form the species $\Sym(\cF)$ of symmetries of $\cF$. The {\em cycle index sum} of $\cF$ is given by 
\[
Z_\cF = \sum_{(F, \sigma)} w_{(F, \sigma)}
\] with the sum index $(F, \sigma)$ ranging over the set $\bigcup_{n \in \ndN_0} \Sym(\cF)[n]$. The reason for studying cycle index sums is the following remarkable property.
\begin{lemma}[{\cite[Sec. 3]{MR633783}}]
	\label{le:oneton}
	Let $U$ be a finite $n$-element set. For any unlabelled $\cF$-object $m$ of size $n$ there are precisely $n!$ symmetries $(F, \sigma)\in \Sym(\cF)[U]$ having the property that $F$ has isomorphism type $m$.
\end{lemma}
 From a probabilistic viewpoint, this observation guarantees that the isomorphism type of the first coordinate of a uniformly at random drawn element from $\Sym(\cF)([n])$ is uniformly distributed among all $n$-element unlabelled $\cF$-objects. 
Lemma~\ref{le:oneton} implies that the ordinary generating series and the cycle index sum are related by 
\[
\tilde{\cF}(z) = Z_{\cF}(z, z^2, z^3, \ldots).
\]
See also {\cite[Sec. 3, Prop. 9]{MR633783}}.

The cycle index sum $Z_\Set$ is easily calculated: For any integer $n \ge 0$ let $\cS_n$ denote the symmetric group of order $n$. Then
\begin{align}
\label{eq:set}
Z_{\Set} = \sum_{n=0}^\infty \frac{1}{n!}\sum_{\sigma \in \cS_n} x_1^{\sigma_1} x_2^{\sigma_2} \cdots x_n^{\sigma_n}.
\end{align}
For any permutation $\sigma$ let $(\sigma_1, \sigma_2,  \ldots) \in \ndN_0^{(\ndN)}$ denote its {\em cycle type}. Then to each element $m=(m_i)_i \in \ndN_0^{(\ndN)}$ correspond only permutations of order $n := \sum_{i=1}^{\infty} i m_i$ and their number is given by $n! / \prod_{i=1}^{\infty} ( m_i! \, i^{m_i})$.  Hence we have
\begin{align*}
Z_{\Set}  = \sum_{m \in \ndN_0^{(\ndN)}} \prod_{i=1}^{\infty}  \frac{x_i^{m_i}} { m_i! \, i^{m_i}} = \prod_{i=1}^\infty \sum_{m_i=0}^\infty  \frac{x_i^{m_i}} { m_i! \, i^{m_i}} = \prod_{i=1}^\infty \exp \left(\frac{x_i} {i} \right) = \exp\left( \sum_{i=1}^\infty \frac{x_i}{i}\right).
\end{align*}
If $(x_i)_i$ would denote a sequence of sufficiently fast decaying positive real-numbers, then this calculation could easily be justified. But they denote a countable set of formal variables, and hence one has every right to ask for a rigorous justification of this argument, in particular why the involved infinite products of formal variables vanish. A correct formalization is to define a topology on the set of power series and interpret these infinite products as actual limits with respect to this topology. We refer the inclined reader to \cite[Appendix A.5]{MR2483235} for an adequate discussion of these questions.

\subsection{Operations on combinatorial species}
\label{sec:cp4}
The framework of combinatorial species offers a large variety of constructions that create new species from others. In the following let $\cF$, $(\cF_i)_{i \in I}$ and $\cG$ denote species and $U$ an arbitrary finite set.
The {\em sum} $\sum_{i \in I} \cF_i$ is defined by the disjoint union
\[
(\sum\nolimits_i \cF_i)[U] = \bigsqcup\nolimits_i \cF_i[U]
\]
if the right hand side is finite for all finite sets $U$. The {\em product} $\cF \cdot \cG$ is defined by the disjoint union
\[
(\cF \cdot \cG)[U] = \bigsqcup_{{\substack{ (U_1, U_2) \\ U_1 \cap U_2 = \emptyset, U_1 \cup U_2 = U}}} \cF[U_1] \times \cG[U_2] \\
\]
with componentwise transport. Thus, $n$-sized objects of the product are pairs of $\cF$-objects and $\cG$-objects whose sizes add up to $n$. If the species $\cG$ has no objects of size zero, we can form the {\em substitution} $\cF \circ \cG$ by
\[
(\cF \circ \cG)[U] = \bigsqcup_{{\substack{ \pi \text{ partition of } U}}} \cF[\pi] \times \prod_{Q \in \pi} \cG[Q]. \\
\]
An object of the substitution may be interpreted as an $\cF$-object whose labels are substituted by $\cG$-objects. The transport along a bijection $\sigma$ is defined by applying the induced map \[\overline{\sigma}:\, \pi \to \{\sigma(Q) \mid Q \in \pi\}, \quad Q \mapsto \sigma(Q)\] of partitions to the $\cF$-object, and the restricted maps $\sigma|_Q$ with $Q \in \pi$ to their corresponding $\cG$-objects. We will often write $\cF(\cG)$ instead of $\cF \circ \cG$. 
Explicit formulas for the generating series and cycle index sums of the discussed constructions are summarized in Table~\ref{tb:egs}. 

{
	\small
	\renewcommand{\arraystretch}{1.2}
	\setlength{\extrarowheight}{2.5 pt}
	\begin{table}
		\begin{center}
			\begin{tabular}{  | l | l | l |}
				\hline

				 & OGF & Cycle index sum \\ \hline
				$\sum_i \cF_i$  & $\sum_i \tilde{\cF}_i(x)$ & $\sum_i Z_{\cF_i}(x_1, x_2, \ldots)$\\
				$\cF \cdot \cG$  & $\tF(x) \tG(x)$ & $Z_\cF(x_1, x_2, \ldots) Z_\cG(x_1, x_2, \ldots)$ \\
				$\cF \circ \cG$  & $Z_\cF(\tG(x), \tG(x^2), \ldots)$ & $Z_\cF(Z_\cG(x_1, x_2, \ldots), Z_\cG(x_2, x_4, \ldots), \ldots)$ \\

				\hline
			\end{tabular}
		\end{center}
		\caption{Relation between combinatorial constructions and generating series.}
		\label{tb:egs}
	\end{table}
}

\subsection{Decomposition of symmetries of the substitution operation}
\label{sec:op1}
\label{sec:symmetry}
\label{SeFrsec:oponsp}
We are going to need a basic understanding of the structure of the symmetries of the composition $\cF \circ \cG$. The following is a summary of a standard decomposition given in \cite[Sec. 2.6.2]{MR2810913}, \cite[Section 3]{MR633783} and \cite[Section 4.3]{MR1629341}. Let $U$ be a finite set. Any element of $\Sym(\cF \circ \cG)[U]$ consists of the following objects: a partition $\pi$ of the set $U$, an $\cF$-structure $F \in \cF[\pi]$, a family of $\cG$-structures $(G_Q)_{Q \in \pi}$ with $G_Q \in \cG[Q]$ and a permutation $\sigma: U \to U$. We require the permutation $\sigma$ to permute the partition classes and induce an automorphism $\bar{\sigma}: \pi \to \pi$ of the $\cF$-object $F$. Moreover, for any partition class $Q \in \pi$ we require that the restriction $\sigma|_Q: Q \to \sigma(Q)$ is an isomorphism from $G_Q$ to $G_{\sigma(Q)}$. For any cycle $\bar{\tau} = (Q_1, \ldots, Q_\ell)$ of $\bar{\sigma}$ it follows that for all $i$ we have $\sigma^\ell(Q_i) = Q_i$ and the restriction $\sigma^\ell|_{Q_i}: Q_i \to Q_i$ is an automorphism of $G_{Q_i}$. 

Conversely, if we know $(G_{Q_1}, \sigma^\ell|_{Q_1})$ and the maps $\sigma|_{Q_i} = (\sigma|_{Q_1})^i$ for $1 \le i \le \ell-1$, we can reconstruct the $\cG$-objects $G_{Q_2}, \ldots, G_{Q_\ell}$ and the restriction $\sigma|_{Q_1 \cup \ldots \cup Q_\ell}$. Here any $k$-cycle $(a_1, \ldots, a_k)$ of the permutation $\sigma^\ell|_{Q_1}$ corresponds to the $k\ell$-cycle \[(a_1, \sigma(a_1), \ldots, \sigma^{\ell-1}(a_1), a_2, \sigma(a_2), \ldots, \sigma^{\ell-1}(a_2), \ldots, a_k, \sigma(a_k), \ldots, \sigma^{\ell-1}(a_k))\] of $\sigma|_{Q_1 \cup \ldots \cup Q_\ell}$. Thus any cycle $\nu$ of $\sigma$ corresponds to a cycle of the induced permutation $\bar{\sigma}$ whose length is a divisor of the length of $\nu$. 

Note that the maps $\sigma|_{Q_i}$ carry information about the labelling, but not really about the structure of the symmetry, as all $\cG$-structure pertaining to a common cycle need to be isomorphic anyway. Up to relabelling, an $\cF \circ \cG$ is already fully described by its induced $\cF$-symmetry and a family of $\cG$-symmetries, one for each cycle of the $\cF$-symmetry: 

\begin{proposition}
	\label{pro:cons}
	If we are given an $\cF$-symmetry $(m, \sigma_m)$ and for each of its cycles $c$ a $\cG$-symmetry $(G_c, \sigma_c)$, then there is a canonical way to assemble an $\cF \circ \cG$ symmetry out of these objects.
\end{proposition}

The details of the construction are as follows. For each cycle $c$ of $\sigma_m$ let $Q_c$ denote the label set of the $\cG$-object $G_c$. For every atom $e$ of the cycle $c$ set $Q_e := Q_c \times \{e\}$ and $(G_{Q_e}, \sigma_{Q_e}) := \Sym(\cG)[f_e](G_c, \sigma_c)$ with $f_e: Q_c \to Q_e$ the canonical bijection. For any label $e$ of the $\cF$-structure $m$ set $f(e) := Q_e$ and let $\pi$ denote the set of all sets $Q_e$. Thus $F := \cF[f](m)$ is an $\cF$-structure with label set $\pi$ and $C := (\pi, F, (G_{Q})_{Q \in \pi})$ is an $\cF \circ \cG$-structure. Let $c$ be a cycle of $\sigma_m$ and $\nu$ a cycle of $\sigma_c$. Fix  an atom $b = b(c)$ of $c$ and an atom $a = a(\nu)$ of $\nu$. Let $\ell$ denote the length of $c$ and $k$ the length of $\nu$. Form the composed cycle  by
\[
((a,b), \ldots, (a, c^{\ell-1}(b)), (\nu(a),b), \ldots, (\nu(a), c^{\ell-1}(b)), \ldots, (\nu^{k-1}(a), b), \ldots, (\nu^{k-1}(a), c^{\ell-1}(b))).
\]
Then the product $\sigma$ of all composed cycles (formed by all choices of $c$ and $\nu$) is an automorphism of the $\cF \circ \cG$-structure $C$. The composed cycles are pairwise disjoint, hence it does not matter in which order we take the product. Note that $\sigma$ does not depend on the choice of the $a$'s but  different choices of the $b$'s result in a different automorphism $\sigma$. More precisely, if for a given cycle $c$ of $\sigma_m$ we choose $c(b)$ instead of $b$, then the resulting automorphism is given by the conjugation $(\text{id}, c) \sigma (\text{id}, c)^{-1}$ instead of $\sigma$. But $(\text{id}, c)$ is an automorphism of the $\cF \circ \cG$-structure $C$, hence the resulting symmetry $(C, (\text{id}, c) \sigma (\text{id}, c)^{-1})$ is isomorphic to $(C, \sigma)$. This implies that the isomorphism type of $(C, \sigma)$ does not depend on the choices of the $a$'s and $b$'s.




\section{Cycle pointing}
\label{sec:op2}
\label{sec:cycle}
\subsection{The cycle pointing operator}
Bodirsky, Fusy, Kang and Vigerske \cite{MR2810913} introduced the cycle pointing operator which maps a species $\cG$ to the species $\cG^\circ$ such that the $\cG^\circ$-objects over a set $U$ are pairs $(G, \tau)$ with $G \in \cG[U]$ and $\tau$ a {\em marked} cycle of an arbitrary automorphism of $G$. Here we count fixpoints as $1$-cycles. The transport is defined by $\sigma.(G, \tau) = (\sigma.G, \sigma \tau \sigma^{-1})$. Any subspecies $\cS \subset \cG^\circ$ is termed {\em cycle-pointed}. The {\em symmetric} cycle-pointed species $\cG^\circledast \subset \cG^\circ$ is defined by restricting to pairs $(G, \tau)$ with $\tau$ a cycle of length at least $2$.

A {\em rooted symmetry} of the cycle-pointed species $\cS \subset \cG^\circ$ is a quadruple $((G, \tau), \sigma, v)$ such that $(G, \tau)$ is a $\cS$-object, $\sigma$ is an automorphism of $G$, $\tau$ is a cycle of $\sigma$ and $v$ is an atom of the cycle $\tau$. Its {\em weight monomial} is given by \[w_{((G, \tau), \sigma, v)} = \frac{t_\ell}{s_\ell} w_{(G, \sigma)}(s_1, s_2, \ldots)\] with $w_{(G, \sigma)}$ denoting the weight of the symmetry $(G, \sigma)$ and $\ell$ the length of the marked cycle $\tau$. We may form the species $\RSym(\cS)$ of rooted symmetries of $\cS$. The pointed cycle index sum of $\cS$ is given by \[\bar{Z}_\cS(s_1, t_1; s_2, t_2; \ldots) = \sum_{(G, \tau, \sigma, v)} w_{(G, \tau, \sigma, v)} \in \ndQ[[s_1, t_1; s_2, t_2; \ldots]]\] with the sum index ranging over the set $\bigcup_{n \in \ndN_0} \RSym(\cS)[n]$.

Let $\cG^\circ_{(\ell)} \subset \cG^\circ$ denote the subspecies given by all cycle pointed objects whose marked cycle has length $\ell$. It follows from the definition of the pointed cycle index sum that
\[
\bar{Z}_{\cG^\circ_{(\ell)}} = \ell t_{\ell} \frac{\partial}{\partial s_\ell} Z_\cG.
\]
Since $\cG^\circ = \sum_{\ell=1}^\infty \cG^\circ_{(\ell)}$ it follows that
\[
\bar{Z}_{\cG^\circ} = \sum_{\ell =1}^\infty \ell t_\ell \frac{\partial}{\partial s_\ell} Z_\cG \quad \text{and} \quad \bar{Z}_{\cG^\circledast} = \sum_{\ell =2}^\infty \ell t_\ell \frac{\partial}{\partial s_\ell} Z_\cG 
.\]

\begin{lemma}[{\cite[Lem. 14]{MR2810913}}]
	\label{le:oneton2}
	Let $U$ be a finite set with $n$ elements and fix an arbitrary linear order on $U$.
	\begin{enumerate}[\qquad 1)]
		\item  The following map is bijective:
		\begin{align*}
		\RSym(\cS)[U] &\to \Sym(\cS)[U], \\
		M=((G,\tau), \sigma,v) &\mapsto ((\tau^{1-\ell(M)}.G, \tau), \sigma \tau^{\ell(M)-1})
		\end{align*}
		with $\ell(M)$ defined as follows: let $k$ denote the length of the cycle $\tau$ and $u$ its smallest atom. Let $0 \le \ell(M) \le k-1$ be the unique integer satisfying $v = \tau^{\ell(M)}.u$. 
		\item
		Any unlabelled cycle-pointed $\cS$-object $m$ of size $n$ corresponds to precisely $n!$ rooted c-symmetries from $\RSym(\cS)[U]$ having the property that the isomorphism type of the underlying $\cS$-object equals $m$.
	\end{enumerate}
\end{lemma}
In particular, the pointed cycle index sum relates to the ordinary generating series by
\[
\tilde{\cS}(x) = \bar{Z}_{\cS}(x,x; x^2, x^2; \ldots).
\]
Moreover, if we draw an element from $\RSym(\cS)[n]$ uniformly at random, then the isomorphism class of the corresponding cycle pointed structure is uniformly distributed among all unlabelled cycle-pointed $\cS$-objects of size $n$. The main point of the cycle-pointing construction is evident from the following fact.
\begin{lemma}[{\cite[Thm. 15]{MR2810913}}]
	\label{le:cyc}
	Any unlabelled $\cG$-structure $m$ of size $n$ may be cycle-pointed in precisely $n$ ways, that is, there exist precisely $n$ unlabelled $\cG^\circ$-structures with corresponding $\cG$-structure~$m$. 
\end{lemma}
Considered from a probabilistic viewpoint, this means that if we draw an unlabelled $\cG^\circ$-structure of size $n$ uniformly at random, then the underlying $\cG$-object is also uniformly distributed. Moreover, Lemma~\ref{le:oneton2} tells us that in order to sample the $\cG^\circ$-object we may sample a rooted symmetry of this size uniformly at random. 

Studying the random $\cG^\circ$-object might be easier due to the additional information given by the marked cycle.
Moreover, Lemma~\ref{le:cyc} implies that \[\tilde{\cG}^\circ(z) = z \frac{\text{d}}{\text{d}z} \tilde{\cG}(z).\]
The pointed cycle index sum of the species $\Set$ is given by \begin{align}
\label{eq:setpoint}
\bar{Z}_{\Set^\circ} = \sum_{\ell=1}^\infty  \ell t_{\ell} \frac{\partial}{\partial s_\ell} Z_\Set(s_1,s_2,\ldots) =\exp \left(\sum_{i=1}^\infty s_i / i  \right)\sum_{\ell = 1}^\infty t_\ell.
\end{align}

\subsection{Operations on cycle pointed species}
Cycle pointed species come with a set of new operations introduced in \cite{MR2810913}. If $\cS \subset \cG^\circ$ is a cycle-pointed species and $\cH$ a species, then the {\em pointed product} $\cS \star \cH$ is the subspecies of $(\cG \cdot \cH)^\circ$ given by all cycle-pointed objects such that the marked cycle consists of atoms of the $\cG$-structure and the $\cG$-structure together with this cycle belongs to $\cS$. The corresponding pointed cycle index sum is given by \[\bar{Z}_{\cS \star \cH} = \bar{Z}_{\cS} Z_\cH.\] 

The cycle-pointing operator obeys the following product rule
\[
(\cG \cdot \cH)^\circ \simeq \cG^\circ \star \cH + \cH^\circ \star \cG.
\]

If $\cH[\emptyset]=\emptyset$ we may form the {\em pointed substitution} $\cS \circledcirc \cH \subset (\cG \circ \cH)^\circ$ as follows. Any $(\cG \circ \cH)^\circ$-structure $P$ has a marked cycle $\tau$ of some automorphism $\sigma$. By the discussion in Section~\ref{sec:symmetry}, this cycle corresponds to a cycle on the $\cG$-structure of $P$ which does not depend on the choice of $\sigma$. Hence the $\cG$-structure of $P$ is cycle-pointed and we say $P$ belongs to $\cS \circledcirc \cH$ if and only if this cycle pointed $\cG$-structure belongs to $\cS$. The corresponding pointed cycle index sum is given by
\begin{align*}
\bar{Z}_{\cS \circledcirc \cH} = \bar{Z}_{\cS}(& Z_\cH(s_1, s_2, \ldots), \bar{Z}_{\cH^\circ}(s_1, t_1; s_2, t_2; \ldots); \\ & Z_\cH(s_2, s_4, \ldots), \bar{Z}_{\cH^\circ}(s_2, t_2; s_4, t_4; \ldots); \ldots).
\end{align*}

\section{(P\'olya-)Boltzmann samplers}
\label{sec:Boltzmann}

Boltzmann samplers were introduced in  \cite{MR2062483, MR2095975, MR2498128} and generalized to P\'olya--Boltzmann samplers in \cite{MR2810913}. We briefly discuss the background to the extend required in our proofs. 

\subsection{Boltzmann models}
The {\em P\'olya--Boltzmann model} was introduced in \cite{MR2810913}: Suppose that we are given a sequence of real numbers $s_1, s_2, \ldots \ge 0$ such that $0 < Z_{\cF}(s_1, s_2, \ldots) < \infty$. Then we may consider the probability distribution on the set $\bigcup_{n=0}^\infty \Sym(\cF)[n]$ that assigns the probability weight
\[
w_{(F, \sigma)} Z_{\cF}(s_1,s_2, \ldots)^{-1} = \frac{s_1^{\sigma_1} s_2^{\sigma_2} \cdots}{n!} Z_{\cF}(s_1,s_2, \ldots)^{-1}
\]
for each $n$ and symmetry $(F, \sigma) \in \Sym(\cF)[n]$. Here $\sigma_i$ denotes the number of $i$-cycles of the permutation $\sigma$. The corresponding {\em P\'olya--Boltzmann sampler} is denoted by $\Gamma Z_{\cF}(s_1, s_2, \ldots)$, and simply refers to a random variable following this distribution, possibly with a description on how to sample it. When describing a sampling procedure the pseudo-code notation \begin{align} \label{eq:not}
(F, \sigma) \leftarrow \Gamma Z_{\cF}(s_1, s_2, \ldots) \end{align}
means that we let $(F, \sigma)$ denote a random $\cF$-symmetry that is independent from all previously considered random variables and sampled according to a P\'olya--Boltzmann distribution for the species $\cF$ with parameters $(s_i)_i$.

\begin{remark}
	\label{re:sampler1}
In the special case $(s_i)_i = (x^i)_i$ for some $x >0$, for each fixed $n$ it holds that all outcomes with size $n$ are equally likely. This means that $\Gamma Z_{\cF}(x, x^2, \ldots)$ conditioned on having a given deterministic size $n$ follows the uniform distribution. By Lemma~\ref{le:oneton} the $n$-sized symmetries from $\Sym(\cF)[n]$ are in a $n:1$ relation to the unlabelled $n$-sized $\cF$-objects. Thus, the $\cF$-object corresponding to the conditioned P\'olya Boltzmann sampler is uniformly distributed among all $n$-sized $\cF$-objects. 
\end{remark}

A P\'olya--Boltzmann model for random cycle pointed species is given by a probability measure on random rooted symmetries: Let $\cS$ be a cycle-pointed species. Given real non-negative numbers $(s_i, t_i)_{i \ge 1}$ such that $0<\bar{Z}_\cS(s_1,t_1; s_2, t_2; \ldots) < \infty$ we may consider the probability measure on the set $\bigcup_{n=0}^\infty \RSym(\cS)[n]$ that assigns probability weight
\[
w_{((G, \tau), \sigma,v)} \bar{Z}_{\cS}(s_1,t_1; s_2, t_2; \ldots)^{-1} = \frac{t_{\ell} s_1^{\sigma_1} \cdots s_{\ell-1}^{\sigma_{\ell-1}} s_{\ell}^{\sigma_{\ell}-1} s_{\ell +1}^{\sigma_{\ell +1}} s_{\ell + 2} ^{\sigma_{\ell +2}} \cdots}{n!\bar{Z}_{\cS}(s_1,t_1; s_2, t_2; \ldots)}
\]
for each $n$ to each rooted symmetry $((G, \tau), \sigma, v) \in \RSym(\cS)[n]$. Here $\ell$ denotes the lengths of the marked cycle $\tau$. The corresponding P\'olya--Boltzmann sampler of this model is denoted by $\Gamma \bar{Z}_\cS(s_1, t_1; s_2, t_2; \ldots)$, and we use a similar notation as in \eqref{eq:not} when describing sampling procedures.

\begin{remark}
	\label{re:sampler2}
	In the special case $(s_i, t_i)_i = (x^i, x^i)_i$ for some $x>0$, for each fixed $n$ we have that all outcomes with size $n$ are equally likely. Hence conditioning $\Gamma \bar{Z}_\cS(x,x; x^2, x^2; \ldots)$ on having size $n$ yields the uniform distribution on $\RSym(\cF)[n]$. By Lemma~\ref{le:oneton2} we know that the rooted symmetries from $\RSym(\cF)[n]$ are in an $n : 1 $ relation to the unlabelled $n$-sized cycle-pointed $\cS$-objects. Thus, the $\cS$-object corresponding to the conditioned P\'olya--Boltzmann sampler follows the uniform distribution among all $n$-sized cycle pointed $\cS$-objects.
\end{remark}

\subsection{Rules for the construction of Boltzmann samplers}
\label{sec:boltzconstruction}

The sampling procedures described in the present exposition were established in \cite[Prop. 38, Prop. 43]{MR2810913}.


\subsubsection{P\'olya--Boltzmann samplers}
\label{sec:pobosa}

Let $\cF$ denote a species and $(s_i)_{i \ge 1}$ non-negative real numbers such that 
\[
	0 < Z_{\cF}(x_1, x_2, \ldots) < \infty.
\]

\subsection*{Products}
Suppose that $\cF = \cF_1 \cdot \cF_2$ is the product of two species $\cF_1$ and $\cF_2$. Then for any finite set $U$ there is a bijection between the set $\Sym(\cF)[U]$ and pairs $(S_1, S_2)$ such that $S_i$ is an $\cF_i$-symmetry for all $i$ and the label sets of the $S_i$ partition the set $U$. This is due to the fact, that given an $\cF$-symmetry $((F_1, F_2), \sigma) \in \Sym(\cF)[U]$ the permutation $\sigma$ must leave the label set $Q_i$ of the $\cF_i$-object $F_i$ invariant and satisfy $\sigma|_{Q_i} .F_i = F_i$, that is $(F_i, \sigma|_{Q_i}) \in \Sym(\cF_i)[Q_i]$. The following pseudo-code procedure is a P\'olya--Boltzmann sampler for the species $\cF$.

\begin{enumerate}[\qquad 1.]
	\item For $i=1,2$ set
	\[
	S_i \leftarrow \Gamma Z_{\cF_i}(s_1, s_2, \ldots).
	\] By the bijection for the symmetries of products, the pair $(S_1, S_2)$ corresponds to an $\cF$-symmetry $(F, \sigma)$ over the (exterior) disjoint union $U$ of the label-sets of the $S_i$. 
	\item Make a uniformly at random choice for a bijection $\nu$ from $U$ to the set of integers $[n]$ with $n$ denoting the size of $U$. Return the relabelled symmetry 
	\[
	\nu.(F,\sigma) = (\nu.F, \nu \sigma \nu^{-1}).
	\]
\end{enumerate}

\subsection*{Substitution}
Suppose that $\cF = \cG \circ \cH$ with $\cH[\emptyset] = \emptyset$ is the composition of a species $\cG$ with another species $\cH$. The symmetries of the substitution were discussed in detail in Section~\ref{sec:symmetry}. The following procedure is a P\'olya--Boltzmann sampler for $\cF$.
\begin{enumerate}[\qquad 1.]
	\item Set
	\[
	(G, \sigma) \leftarrow \Gamma Z_{\cG}(Z_{\cH}(s_1, s_2, \ldots), Z_{\cH}(s_2, s_4, \ldots), \ldots).
	\]
	That is, let $(G, \sigma)$ denote a random $\cG$-symmetry that follows a P\'olya--Boltzmann distribution with parameters $Z_{\cH}(s_1, s_2, \ldots), Z_{\cH}(s_2, s_4, \ldots), \ldots$.
	\item For each cycle $\tau$ of $\sigma$ let $|\tau|$ denote its lengths and set
	\[
	(H_\tau, \sigma_\tau) \leftarrow \Gamma Z_{\cH}(s_{|\tau|}, s_{2|\tau|}, \ldots).
	\]
	That is, the symmetries $(H_\tau, \sigma_\tau)$, $\tau$ cycle of $\sigma$, are independent (conditional on $\sigma$) and follow P\'olya--Boltzmann distributions.
	\item For each cycle $\tau$, make $|\tau|$ identical copies copies of $(H_\tau, \sigma_\tau)$ and assemble an $\cF$-symmetry $(F, \gamma)$ out of $(G, \sigma)$ and the copies of the $(H_\tau, \sigma_\tau)$ as described in Proposition~\ref{pro:cons}.
	\item Choose bijection $\nu$ from the vertex set of $(F, \gamma)$ to an appropriate sized set of integers $[n]$ and return the relabelled symmetry
	\[
	\nu.(F, \gamma) = (\nu.F, \nu \gamma \nu^{-1}).
	\]
\end{enumerate}

\subsubsection{P\'olya--Boltzmann samplers for cycle-pointed species}
In the following, we suppose that $\cF$ is a cycle pointed species and  that $s_1, t_1, s_2,t_2, \ldots$ are non-negative real numbers such that \[
0 < \bar{Z}_\cF(s_1,t_1; s_2,t_2; \ldots) < \infty.
\]

\label{sec:pobocyc}

\subsection*{Cycle pointed products}
Suppose that $\cF = \cG \star \cH$ with $\cG$ a cycle-pointed species and $\cH$ a species. Then for any finite set $U$ there is a canonical choice for a bijection between the set $\RSym(\cF)[U]$ and tuples $(S_1,S_2)$ with $S_1$ a rooted symmetry of $\cG$, $S_2$ a symmetry of $\cG$, such that the label sets of $S_1$ and $S_2$ form a partition of $U$. 
The following procedure is a P\'olya--Boltzmann sampler for $\cF$.

\begin{enumerate}[\qquad 1.]
	\item Set
	\[
	S_1 \leftarrow \Gamma \bar{Z}_{\cG}(s_1,t_1; s_2,t_2; \ldots).
	\]
	\item Set
	\[
	S_2 \leftarrow \Gamma Z_{\cH}(s_1, s_2, \ldots).
	\]
	\item Let $U$ denote the exterior disjoint union of the label sets of $S_1$ and $S_2$. The tupel $(S_1,S_2)$ corresponds to a rooted symmetry $S$ over the set $U$.
	\item Make a uniformly at random choice of a bijection $\nu$ from $U$ to the set of integers $[n]$ with $n$ denoting the size of $U$. Return the relabelled rooted symmetry 
	$
	\nu.S$.
\end{enumerate}

\subsection*{Cycle pointed substitution}
Suppose that $\cF = \cG \circledcirc \cH$ with $\cG$ cycle-pointed and $\cH[\emptyset] = \emptyset$. The symmetries of the substitution were discussed in detail in Section~\ref{sec:symmetry}. The following procedure is a P\'olya--Boltzmann sampler for $\cF$.
\begin{enumerate}[\qquad 1.]
	\item Set 
	\[
	((G, \tau_0), \sigma,v_0) \leftarrow \Gamma \bar{Z}_{\cG}(h_1, \bar{h}_1; h_2, \bar{h}_2; \ldots) 
	\]
	with parameters
	\[
	h_i = Z_{\cH}(s_i, s_{2i}, \ldots) \quad \text{and} \quad  \bar{h}_i = \bar{Z}_{\cH^\circ}(s_i,t_i; s_{2i},t_{2i}; \ldots).
	\]

	\item For each unmarked cycle $\tau$ of $\sigma$ let $|\tau|$ denote its lengths and set
	\[
	(H_\tau, \sigma_\tau) \leftarrow \Gamma Z_{\cH}(s_{|\tau|}, s_{2|\tau|}, \ldots).
	\]
	\item For the marked cycle $\tau_0$ set
	\[
	( (H_{\tau_0}, c_{\tau_0}), \sigma_{\tau_0}, v_{\tau_0}) \leftarrow \Gamma Z_{\cH^\circ}(s_{|\tau_0|},t_{|\tau_0|};  s_{2|\tau_0|},t_{2|\tau_0|}; \ldots).
	\]

	\item Assemble an $\cF$-symmetry $(F, \gamma)$  out of the $\cG$-symmetry $(G, \sigma)$ and the $\cH$-symmetries $(H_\tau, \sigma_\tau)$ according to the construction of Proposition~\ref{pro:cons}.
	
	Let $c$ denote the cycle that gets composed out of the $|\tau_0|$ copies of the cycle $c_{\tau_0}$ in this construction. The marked vertex $v_{\tau_0}$ has $|\tau_0|$ copies (one for each atom of $\tau_0$) and we let $u$ denote the copy that corresponds to the marked atom $v_0$ of $\tau_0$. Thus
	\[
	((F,c), \gamma,u)
	\]
	is a rooted symmetry of $\cF$.
	\item Choose a bijection $\nu$ from the vertex set of $((F,c), \gamma,u)$ to an appropriate sized set of integers $[n]$ and return the relabelled rooted symmetry
	\[
	\nu.((F,c), \gamma,u) = ((\nu.F, \nu c \nu^{-1}), \nu \gamma \nu^{-1}, \nu.u).
	\]
\end{enumerate}

\section{Random multisets}
\label{sec:mult}

If $\cF$ is a species of structures with $\cF[\emptyset] = \emptyset$, then unlabelled $\Set \circ \cF$-objects are termed \emph{multisets}.  They consist of unordered collections of unlabelled $\cF$-objects where each object is allowed to appear multiple times. The following preliminary observation is a consequence of a more general result established by Barbour and Granovsky~\cite[Thm. 2.2]{MR2121024}.

\begin{lemma}
	\label{le:part}
	Suppose that 
	\[
		[z^n] \tilde{\cF}(z) = f(n) n^{-\beta} \rho^n
	\]
	for some constants $\rho >0$ and $\beta >1$, and a function $f$ that varies slowly at infinity. Then the largest component in a uniform  $n$-sized multiset of unlabelled $\cF$-structures has size $n + O_p(1)$.
\end{lemma}

\section{Proof of the main theorems}
\label{sec:mainproof}
{\em
	Throughout this section, let $\Omega$ be a set of positive integers containing the number $1$ and at at least one integer equal or greater than $3$. We let $\cF$ denote the species of unrooted trees and $\cF_\Omega$ its subspecies of trees with vertex degrees in the set $\Omega$. Analogously, we let $\cA$ denote the species of rooted trees and $\cA_{\Omega^*}$ the subspecies of rooted trees with vertex outdegrees in the shifted set $\Omega^* = \Omega - 1$. In the following we will always assume that $n$ denotes an integer satisfying $n \equiv 2 \mod \gcd(\Omega^*)$ and $n$ large enough such that trees with $n$ vertices and vertex degrees in the set $\Omega$ exist. Let $\rho$ denote the radius of convergence of the generating series $\tilde{\cA}_{\Omega^*}(z)$.
	
	We let $(\mT_n, \tau_n)$ denote a random cycle-pointed tree drawn uniformly from the unlabelled $\cF_\Omega^\circ$-objects of size $n$. As discussed in Lemma~\ref{le:oneton}, this implies that $\mT_n$ is the uniform random unlabelled unrooted tree with $n$ vertices and vertex degrees in the set $\Omega$. Moreover, let $\mA_{n-1}$ a random rooted tree drawn uniformly from the unlabelled $\cA_{\Omega^*}$-objects of size $n-1$.

	We let $c_{\Omega^*}>0$ denote the constant from Equation~\eqref{eq:convpolya} such that the uniformly drawn unlabelled rooted tree $\mA_{n-1}$ satisfies
	\[
	(\mA_{n-1}, c_{\Omega^*} n^{-1/2} d_{\mA_{n-1}}) \convdis (\CRT, d_{\CRT})
	\] with respect to the Gromov--Hausdorff metric. Moreover, let $\hat{\mA}_{\Omega^*}$ denote the infinite rooted tree from Equation~\eqref{eq:localconv} with
	\[
		d_{\textsc{TV}}( V_{k_n}(\mA_{n-1}, u_{n-1}), V_{k_n}(\hat{\mA}_{\Omega^*})) \to 0
	\]
	for every sequence $k_n = o(\sqrt{n})$,  with $u_{n-1}$ denoting a uniformly at random selected vertex of $\mA_{n-1}$.
}

\subsection{Decomposition of cycle-pointed trees}

\label{SeFrsec:enumprop}

Given a cycle pointed tree $(T, \tau)$ such that the marked cycle $\tau$ has length at least $2$ we may consider its {\em connecting paths}, i.e. the paths in $T$ that join consecutive atoms of $\tau$. Any such path has a middle, which is either a vertex if the path has odd length, or an edge if the path has even length. All connecting paths have the same lengths and by \cite[Claim 22]{MR2810913} they share the same middle, called the {\em center of symmetry}. See Figure~\ref{SeFrfi:conpa} for an illustration.

\begin{figure}[t]
	\centering
	\begin{minipage}{1\textwidth}
		\centering
		\includegraphics[width=0.6\textwidth]{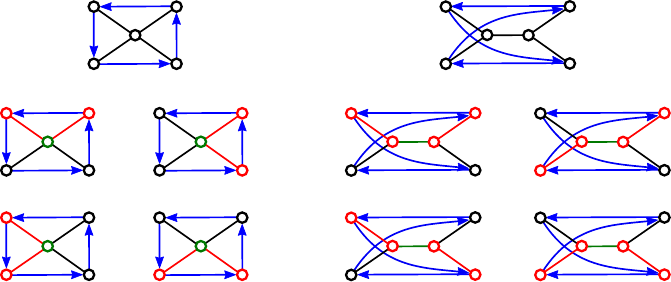}
		\caption{Two unlabelled cycle-pointed trees. The marked cycle is depicted in blue, connecting paths in red, and the cycle-pointing centers in green.}
		\label{SeFrfi:conpa}
	\end{minipage}
\end{figure}

The cycle pointing decomposition given in \cite[Prop. 25]{MR2810913} splits the species $\cF_\Omega^\circ$ into three parts,
\[
\cF_\Omega^\circ \simeq \cX^\circ \star (\Set_\Omega \circ \cA_{\Omega^*}) + \Set_{\{2\}}^\circledast  \circledcirc \cA_{\Omega^*} + (\Set_\Omega^\circledast \circledcirc \cA_{\Omega^*}) \star \cX.
\]
Here \[\cS := \cX^\circ \star (\Set_\Omega \circ \cA_{\Omega^*})\] corresponds to the trees with a marked fixpoint
and the other summands to trees with a marked cycle of length at least two. More specifically, \[\cE := \Set_{\{2\}}^\circledast  \circledcirc \cA_{\Omega^*}\] corresponds to the symmetric cycle pointed trees whose center of symmetry is an edge
and \[\cV := (\Set_\Omega^\circledast \circledcirc \cA_{\Omega^*}) \star \cX\] to those whose center of symmetry is a vertex.


\subsection{Enumerative properties}

We start by collecting some basic enumerative facts. The following preliminary observation summarizes enumerative properties of P\'olya trees with vertex degree restrictions.

\begin{proposition}[{\cite[Prop. 4.1]{2015arXiv150207180P}}] The following statements hold.
	\label{SeFrpro:ser1}
	\label{SeFrpro:cnt}
	\begin{enumerate}[i)]
		\item The radius of convergence $\rho$ of the series $\tilde{\cA}_{\Omega^*}(z)$ satisfies $0 < \rho < 1$ and $\tilde{\cA}_{\Omega^*}(\rho) < \infty$. 
		\item There is a positive constant $d_{\Omega^*}$ such that 
		\[
		[z^m]\tilde{\cA}_{\Omega^*}(z) \sim d_{\Omega^*} m^{-3/2} \rho^{-m}
		\] as the number $m \equiv 1 \mod \gcd(\Omega^*)$ tends to infinity.
		\item For any subset $\Lambda \subset \ndN$ the series
		\[E^{\Lambda}(z,w) = z Z_{\Set_{\Lambda}}(w, \tilde{\cA}_{\Omega^*}(z^2), \tilde{\cA}_{\Omega^*}(z^3), \ldots)\]
		satisfies \[E^\Lambda(\rho + \epsilon, \tilde{\cA}_{\Omega^*}(\rho) + \epsilon) < \infty\] for some $\epsilon > 0$.
	\end{enumerate}
\end{proposition}
In \cite[Prop. 24]{MR2810913} the cycle-pointing decomposition was used in order to provide a new method for determining the asymptotic number of free trees. This may be extended to the case of vertex degree restrictions. A detailed justification is given in Section~\ref{sec:them} below.
\begin{proposition}
	\label{SeFrpro:together}
	\label{SeFrpro:rho}
	\label{SeFrpro:ser3}
	\label{SeFrpro:small}
	The series $\tilde{\cF}_\Omega(z)$ and $\tilde{\cA}_{\Omega^*}(z)$ both have the same radius of convergence $\rho$. Moreover, the following statements hold.
	\begin{enumerate}[i)]
		\item There is a constant $d_{\Omega^*}'$ such that
		\[
		[z^n] \tilde{\cF}_\Omega(z) \sim d_{\Omega*}' \rho^{-n} n^{-5/2}
		\]
		as $n \equiv 2 \mod \gcd(\Omega^*)$ tends to infinity.
		\item For any set $\Lambda \subset \ndN$ the series
		\[
		F^\Lambda(z,w) = \bar{Z}_{\Set_\Lambda^\circledast}(w, \tilde{\cA}_{\Omega^*}^\circ(z); \tilde{\cA}_{\Omega^*}(z^2), \tilde{\cA}_{\Omega^*}^\circ(z^2); \tilde{\cA}_{\Omega^*}(z^3), \tilde{\cA}_{\Omega^*}^\circ(z^3); \ldots )
		\] satisfies $F^\Lambda(\rho + \epsilon, \tilde{\cA}_{\Omega^*}(\rho)+\epsilon) <0$ for some $\epsilon > 0$.
		\item The power series \[\bar{Z}_{\Set_{\{2\}}^\circledast  \circledcirc \cA_{\Omega^*}}(z) = \tilde{\cA}_{\Omega^*}^\circ(z^2)\] has radius of convergence greater than $\rho$.
	\end{enumerate}
\end{proposition}

\subsection{Approximation arguments}

We are going to treat the classes $\cS$, $\cE$, and $\cV$ separately.

\subsubsection{The class $\cE$ of symmetric cycle pointed trees whose center of symmetry is an edge}

The event $(\mT_n, \tau_n) \in \cE$ is so unlikely, that we will be able to neglect this case:

\begin{lemma}
	\label{le:small}
	There are constants $C,c>0$, such that for all $n$
	\[
	\Pr{(\mT_n, \tau_n) \in \cE} \le C \exp(-c n).
	\]
\end{lemma}

Geometrically speaking, this can be explained by the fact that any unlabelled cycle pointed tree from $\cE$ corresponds bijectively to a cycle pointed P\'olya tree from $\cA_{\Omega^*}^\circ$ having precisely half of its size. Compare with Figure~\ref{SeFrfi:typeb}. The number of such objects is roughly given by $\rho^{n/2}$, while the number of all cycle pointed trees in $\cF_\Omega^\circ$ is roughly given by $\rho^n$, which is exponentially larger.

\begin{figure}[h]
	\centering
	\begin{minipage}{1.0\textwidth}
		\centering
		\includegraphics[width=0.4\textwidth]{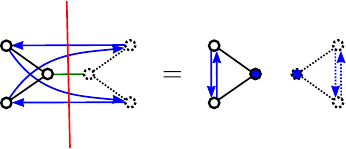}
		\caption{Any unlabelled $\cE = \Set_{\{2\}}^\circledast  \circledcirc \cA_{\Omega^*}$ object corresponds to two identical copies of a cycle-pointed P\'olya tree.}
		\label{SeFrfi:typeb}
	\end{minipage}
\end{figure}

\subsubsection{The class $\cS$ of cycle pointed trees with a marked fixpoint}


\begin{lemma}
	\label{SeFrle:sobj}
	Let $\mS_n$ be drawn uniformly at random from the unlabelled \[\cS = \cX^\circ \star (\Set_\Omega \circ \cA_{\Omega^*})\] objects of size $n$. Then the following properties hold.
	\begin{enumerate}[\qquad a)]
		\item There are constants $C,c > 0$ such that for all $n$ and $x \ge 0$ it holds that \[\Pr{\Di(\mT_n) \ge x} \le C \exp(-c x^2/n).\] 
		\item There is a random number $K_n = n + O_p(1) \le n$ and a coupling of $\mS_n$ with a partition into two rooted subtrees $\mB_n$, $\mC_n$ that intersect only in their roots and satisfy $\mC_n \eqdist \mA_{K_n}$. 
	\end{enumerate}
\end{lemma}

The reason for this is, that each unlabelled $\cS = \cX^\circ \star (\Set_\Omega \circ \cA_{\Omega^*})$ cycle pointed trees corresponds bijectively to  a P\'olya tree, in which each vertex degree must lie in $\Omega$. That is, the outdegree of the root lies in $\Omega$, and the outdegrees of all remaining vertices lie in $\Omega^*$. Compare with Figure~\ref{SeFrfi:typea}.

\begin{figure}[h]
	\centering
	\begin{minipage}{1.0\textwidth}
		\centering
		\includegraphics[width=0.4\textwidth]{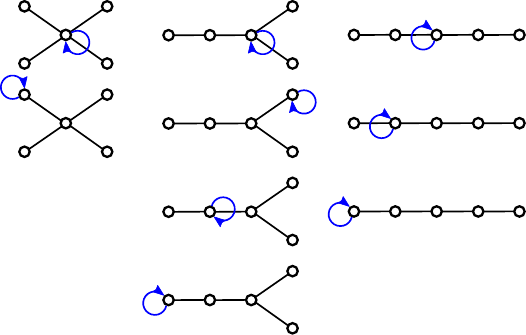}
		\caption{Unlabelled $\cS = \cX^\circ \star (\Set_\Omega \circ \cA_{\Omega^*})$ cycle pointed trees correspond to P\'olya trees, in which each vertex degree must lie in $\Omega$.}
		\label{SeFrfi:typea}
	\end{minipage}
\end{figure}

\subsubsection{The class $\cV$ of symmetric cycle pointed trees whose center of symmetry is a vertex}

\begin{lemma}
	\label{SeFrle:vobj}
	Let $\mV_n$ be drawn uniformly from the unlabelled \[\cV = (\Set_\Omega^\circledast \circledcirc \cA_{\Omega^*}) \star \cX\] objects of size $n$. Then the following statements hold.
	\begin{enumerate}[\qquad a)]
	\item There are constants $C,c>0$ such that for all $x\ge 0$ and $n$ we have the tail bound \[\Pr{\Di(\mV_n) \ge x} \le C \exp(-c x^2/n).\] 
	\item There is a random number $K_n = n + O_p(1) \le n$ and a coupling of $\mV_n$ with a partition into two rooted subtrees $\mB_n$, $\mC_n$ that intersect only in their roots and satisfy $\mC_n \eqdist \mA_{K_n}$. 
	\end{enumerate}
\end{lemma}

The key point is that any unlabelled cycle pointed tree from $\cV$ corresponds to a P\'olya tree $A$ from $\cA_{\Omega^*}$ where each non-root vertex must have outdegrees in $\Omega^*$, together with a number $K$ of identical copies of a symmetrically cycle pointed P\'olya tree $A^\circ$ from $A^\circledast_{\Omega^*}$, such that the sum of the root degrees of $A$ and the $K$ copies of $A^\circ$ lies in $\Omega$. Compare with Figure~\ref{SeFrfi:typec}.

\begin{figure}[h]
	\centering
	\begin{minipage}{1.0\textwidth}
		\centering
		\includegraphics[width=0.55\textwidth]{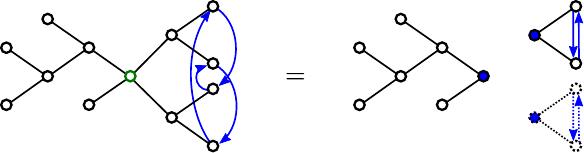}
		\caption{Decomposition of an unlabelled $\cV = (\Set_\Omega^\circledast \circledcirc \cA_{\Omega^*}) \star \cX$ object into a P\'olya tree and a number of identical copies of a symmetrically cycle-pointed P\'olya tree.}
		\label{SeFrfi:typec}
	\end{minipage}
\end{figure}

\subsection{Proof of the main results: Theorems~\ref{SeFrte:main2}, \ref{te:locconv}, and \ref{te:appr}}

Having these results at hand, we may deduce the scaling limit, the Benjamini--Schramm limit and the tail-bound for the diameter for the random unlabelled tree $\mT_n$ by building on the corresponding results for the random P\'olya tree $\mA_{n-1}$.

\begin{proof}[Proof of Theorem~\ref{te:appr}]
	Lemma \ref{le:small} implies that the total variation distance between the unrooted tree $\mT_n$ and a mixture of random $\cS$ and $\cV$ structures is exponentially small. Lemmas \ref{SeFrle:sobj} and \ref{SeFrle:vobj} imply that both $\mS_n$ and $\mV_n$  look like a large randomly sized P\'olya tree with a stochastically bounded rest. Consequently their mixture looks like a large randomly sized P\'olya tree with a small rest which is a mixture of the two stochastically bounded small trees corresponding to $\mS_n$ and $\mV_n$. This completes the proof.
\end{proof}

\begin{proof}[Proof  of Theorem \ref{te:locconv}]
	Theorem~\ref{te:appr} implies that it suffices to study the tree $\bar{\mT}_n$. 
	For the local limit, let $u_n$ denote a uniformly at random drawn vertex of the tree $\bar{\mT}_n$, and let $k_n = o(\sqrt{n})$ denote a given sequence. It is clear that the random vertex $u_n$ lies with high probability in the subtree $\mA_{K_n}$, and that conditioned on this event it is uniformly distributed among its vertices. Note that $K_n = n + O_p(1)$ implies that with high probability $K_n \ge n - \log n \to \infty$ and $k_n= o(\sqrt{n}) = o(\sqrt{K_n})$. By Equation~\eqref{eq:localconv} and $K_n \to \infty$ it follows that the radius $k_n$ neighbourhood of a random vertex in $\mA_{K_n}$ is close in total variation to the $k_n$ neighbourhood of the infinite random tree $\hat{\mA}_{\Omega^*}$, and that a random vertex in $\mA_{K_n}$ has with high probability height strictly larger than $k_n$. In particular, with high probability the neighbourhood does not contain the root-vertex of $\mA_{K_n}$ and is hence not influenced by the small tree $\mB_n$ that gets attached to the root of $\mA_{K_n}$ to form the tree $\bar{\mT}_n$. This readily verifies that
	\[
		d_{\textsc{TV}}( V_{k_n}(\bar{\mT}_n, u_n), V_{k_n}(\hat{\mA}_{\Omega^*})) \to 0,
	\]
	and hence completes the proof.
\end{proof}

\begin{proof}[Proof of Theorem \ref{SeFrte:main2}]
	For the scaling limit, it suffices by Theorem~\ref{te:appr} to consider the tree $\bar{\mT}_n$. As $|\mB| = O_p(1)$ it follows that with high probability it holds that, say, $|\mB| \le n^{1/4}$. Hence
	it holds that 
	\begin{align}
		\label{eq:max}
		d_{\textsc{GH}}( \bar{\mT}_n / \sqrt{n}, \mA_{K_n}/ \sqrt{n}) \convp 0.
	\end{align}
	Note that $\mK_n \convdis \infty$ and Equation~\eqref{eq:convpolya} imply that
	\begin{align}
		\label{eq:moritz}
		c_{\Omega^*} \mA_{K_n} / \sqrt{K_n} \convdis \CRT.
	\end{align}
	In particular, $\Di(\mA_{K_n}) = O_p(\sqrt{K_n})$ and hence
	\[
		d_{\textsc{GH}}(  \mA_{K_n}/ \sqrt{n}, \mA_{K_n} / \sqrt{K_n}) \le O_p(1)(1 - \sqrt{K_n / n}) \convp 0.
	\]
	Together with Equation~\eqref{eq:max} this implies that
	\[
		d_{\textsc{GH}}( \bar{\mT}_n / \sqrt{n}, \mA_{K_n} / \sqrt{K_n}) \convp 0
	\]
	and by the limit in \eqref{eq:moritz} the scaling limit for $\bar{\mT}_n$ follows. The inclined reader may note that the arguments above work just as fine for the Gromov--Hausdorff--Prokhorov metric with respect to the uniform measure on the leaves or all vertices.
	
	For the tail bound of the diameter, note that it suffices to show such a bound for $\Pr{\Di(\mT_n) \ge x}$ when $ x \le n$. By Lemmas~\ref{le:small}, \ref{SeFrle:sobj} and \ref{SeFrle:vobj} it follows that there are constants $C_i, c_i >0$, for $i=1,2,3$, such that
	\begin{align*}
	\Pr{\Di(\mT_n) \ge x} &\le \sum_{\cB \in \{\cE, \cS, \cV \}} \Pr{\Di(\mT_n) \ge x \mid (\mT_n, \tau_n) \in \cB}\,\, \Pr{(\mT_n, \tau_n) \in \cB} \\
	&\le C_1 \exp(-c_1 n) + \sum_{i=2}^{3} C_i \exp(-c_i x^2 / n).
	\end{align*}
	As we assumed that $x \le n$, it holds that
	\[
	\exp(-c_1 n) \le \exp(-c_1 x^2 /n).
	\]
	Hence for a suitable choice of constants $C,c>0$, it follows that
	\[
	\Pr{\Di(\mT_n) \ge x} \le C \exp(-c x^2 /n)
	\]
	for all $n$ and $x \ge 0$.
\end{proof}


\subsection{Proof of the enumerative observation Proposition~\ref{SeFrpro:small}  }
\label{sec:them}

\begin{proof}[Proof of Proposition~\ref{SeFrpro:small}]
	Let $\rho$ denote the radius of convergence of $\tilde{\cA}_{\Omega^*}(z)$.
	Claim iii) follows from the fact that  $\rho < 1$ and that the series \[\tilde{\cA}_{\Omega^*}^\circ(z) = z \frac{\text{d}}{\text{d}z} \tilde{\cA}_{\Omega^*}(z)\] also has radius of convergence~$\rho$. We proceed with claim ii). The series $\bar{Z}_{\Set^\circledast_\Lambda}$ is dominated coefficient-wise by the series
	\[
	\bar{Z}_{\Set^\circledast}(s_1, t_1; s_2, t_2; \ldots) = \exp \left(\sum_{k=1}^\infty s_k / k\right) \sum_{i=2}^\infty t_i
	\]
	and hence $F^\Lambda(z,w)$ is dominated by \[\exp \left(w + \sum_{k=2}^\infty \tilde{\cA}_{\Omega^*}(z^k) / k \right) \sum_{i=2}^\infty \tilde{\cA}_{\Omega^*}^\circ(z^i).\] Since $\rho < 1$ this series is finite for $z= \rho + \epsilon$ and $w = \tilde{\cA}_{\Omega^*}(\rho) + \epsilon$ if $\epsilon>0$ is sufficiently small.  In order prove claim i) we are going to perform a singularity analysis of the series $\tilde{\cF}^\circ_{\Omega}(z)$. The cycle pointing decomposition 
	\[
	\cF_\Omega^\circ \simeq \cX^\circ \star (\Set_\Omega \circ \cA_{\Omega^*}) + \Set_{\{2\}}^\circledast  \circledcirc \cA_{\Omega^*} + (\Set_\Omega^\circledast \circledcirc \cA_{\Omega^*}) \star \cX
	\]
	yields that the series
	$
	\tilde{\cF}_\Omega^\circ(z) = z \frac{\text{d}}{\text{d}z}\tilde{\cF}_\Omega(z)
	$ can be written in the form \[\tilde{\cF}_\Omega^\circ(z) = z h(z, \tilde{\cA}_{\Omega^*}(z))\] with
	\[
	h(z,w) = E^{\Omega}(z,w) + F^\Omega(z,w) + \tilde{\cA}^\circ_{\Omega^*}(z^2)/z. 
	\]
	Here we let $E^\Omega$ be defined as in Proposition~\ref{SeFrpro:cnt}.
	Set $d = \gcd(\Omega^*)$. We have that $\tilde{\cA}_{\Omega^*}(z)$ satisfies the prerequisites of the type of power series studied in Jason, Stanley and Yeats \cite[Thm. 28]{MR2240769}: Its dominant singularities (all of square-root type) are given by the rotated points \[U = \{\omega^k \rho \mid k=0, \ldots, d-1\}\] with \[\omega = e^{\frac{2 \pi i}{d}}.\] Moreover \[\tilde{\cA}_{\Omega^*}(\omega z) = \omega \tilde{\cA}_{\Omega^*}(z)\] for all $z$ in a generalized $\Delta$-region with wedges removed at the points of $U$. 
	We have that $h(z,w)$ is a power series with non-negative coefficients and by claim i) and ii) and Proposition~\ref{SeFrpro:cnt} we have \[h(\tilde{\cA}_{\Omega^*}(\rho) + \epsilon, \rho + \epsilon) < \infty\] for some $\epsilon > 0$. Hence the dominant singularities and their types are driven by the series $\tilde{\cA}_{\Omega^*}(z)$. We may apply a standard result for the singularity analysis of functions with multiple dominant singularities \cite[Thm. VI.5]{MR2483235}  and obtain that 
	\begin{align}
	\label{eq:identical}
	[z^{m}]h(z, \tilde{\cA}_{\Omega^*}(z)) \sim d_{\Omega^*}' m^{-3/2} \rho^{-m}
	\end{align} for $m \equiv 1 \mod \gcd(\Omega^*)$ and $d_{\Omega^*}'>0$ a constant. 
\end{proof}

\subsection{Proofs of the approximation arguments: Lemmas~\ref{le:small}, \ref{SeFrle:sobj}, and \ref{SeFrle:vobj}}

\subsubsection{Cycle pointed trees whose cycle center is an edge}

\begin{proof}[Proof of Lemma~\ref{le:small}]
	The probability for this event is given by the ratio of unlabelled cycle pointed trees of $\cE$ with $n$ vertices, and the unlabelled cycle pointed trees in $\cF_\Omega$ with $n$ vertices. Hence
	\[\Pr{(\mT_n, \tau_n) \in \cE} = \frac{[z^n] \tilde{\cE}(z)}{[z^n] \tilde{\cF}^\circ(z)}.\]
	By Proposition~\ref{SeFrpro:rho}, $iii)$, the radius of convergence of the ordinary generating series $\tilde{\cE}(z)$ is strictly larger than the radius of convergence $\rho$ of $\tilde{\cF}^\circ(z)$. This yields the claim.
\end{proof}

\subsubsection{Cycle pointed trees whose cycle center is a fixpoint}

	It holds  that \[\cS = \cX^\circ \star (\Set_\Omega \circ \cA_{\Omega^*}) \simeq \cX \cdot(\Set_\Omega \circ \cA_{\Omega^*}),\] hence we do not require cycle pointing techniques in this case. Let $(\mS_n, \sigma)$ be drawn uniformly at random from the set $\Sym(\cS)[n]$. Let $\pi_n$ denote the corresponding partition. By the discussion in Section~\ref{SeFrsec:oponsp}, $\sigma$ induces an automorphism \[\bar{\sigma}: \pi_n \to \pi_n\] of the $\Set_\Omega$-object. Moreover, let $F_n \subset \pi_n$ denote the fixpoints of $\bar{\sigma}$, $f_n = |F_n|$ their number and for each fixpoint $Q \in F_n$ let $(\mA_Q, \sigma_Q)$ denote the corresponding symmetry from $\Sym(\cA_{\Omega^*})[Q]$. Let $H_n$ denote the total size of the trees dangling from cycles with length at least $2$. We are going to make the following observations.
	
	\begin{lemma}
		\label{le:pre11}
		The following statements hold.
	\begin{enumerate}[1)]
		\item There are constants $C_1>0$ and $0<\gamma<1$ such that for all $n$ and $x \ge 0$ we have that \[\Pr{H_n \ge x} \le C_1 n^{3/2} \gamma^x\] and \[\Pr{f_n \ge x} \le C_1 n^{3/2} \gamma^x.\]
		\item The maximum size  of the individual trees corresponding to the fixpoints of $\bar{\sigma}$ satisfies \[
		\max_{Q \in F_n}|\mA_Q| = n + O_p(1).
		\]
		\item There is a constant $C_2>0$ such that \[\Ex{f_n} \le C_2\] for all $n$.
	\end{enumerate}
	\end{lemma} 
	This is sufficient to prove Lemma~\ref{SeFrle:sobj}:
	
	\begin{proof}[Proof of Lemma \ref{SeFrle:sobj}]
	We start with claim a),  the tail bound for the diameter. First, it suffices to show such a bound for all $\sqrt{n} \le x \le n$. If $\Di(\mS_n) \ge x$, then we have $H_n \ge x/2$ or $\max_{Q \in F_n} \He(\mA_Q) \ge x/2 -1$. By 1), we have \[\Pr{H_n \ge x/2} \le C_1 n^{3/2}\gamma^{x/2}\] and  there are constants $C_4, c_4 > 0$ such that \[C_1 n^{3/2}\gamma^{x/2} \le C_4 \exp(-c_4 x^2/n)\] for all $n$ and $\sqrt{n} \le x \le n$. Let $\mathfrak{E}_n$ denote the event $\max_Q \He(\mA_Q) \ge x/2 -1$. It holds that
	\[
	\Pr{\mathfrak{E}_n} \le \sum_{F} \Pr{F_n = F} \Pr{\mathfrak{E}_n \mid F_n = F}.
	\]
	with $F$ ranging over all subsets of partitions of $[n]$ with $\Pr{F_n = F}>0$. By the discussion of symmetries in Section~\ref{SeFrsec:oponsp} we have that given $F_n = F$, the symmetries $(\mA_Q, \sigma_Q)_{Q \in F}$ are independent and for each $Q \in F$ we have that $(\mA_Q, \sigma_Q)$ gets drawn uniformly at random from the set $\Sym(\cA_{\Omega^*})[Q]$. That is, $\mA_Q$ gets drawn uniformly at random from all unlabelled P\'olya trees with outdegrees in the set $\Omega^*$. By Inequality~\eqref{SeFrle:tailrooted} it follows that there are positive constants $C_5, c_5$ such that uniformly for all $n$ and $x$
	\[
	\Pr{\mathfrak{E}_n \mid F_n = F} \le C_5 \sum_{Q \in F} \exp(-c_4 x^2/|Q|) \le |F| C_4 \exp(-c_5 x^2/n).
	\]
	It follows that \[\Pr{\cE_n} \le C_5 \exp(-c_5 x^2/n) \sum_{F} \Pr{F_n = F} |F| \le \Ex{f_n} C_5 \exp(-c_5 x^2/n).\] By 3) we have that \[\Ex{f_n} \le C_2\] for all $n$. Thus, for some $C_6, c_6 > 0$, it holds that
	\[
	\Pr{\Di(\mS_n) \ge x} \le C_4 \exp(-c_4 x^2/n) + C_2 C_5\exp(-c_5 x^2/n) \le C_6 \exp(-c_6 x^2/n)
	\]
	uniformly for all $n$ and $\sqrt{n} \le x \le n$. Thus the claims 1) and 3) of Lemma~\ref{le:pre11} imply the tail bound for the diameter. 
	
	We continue with claim b), the approximation argument. Select one of the partition classes from $F_n$ with maximal size uniformly at random and let $\mX_n$ denote the corresponding tree. Note that by the substitution rule for Boltzmann distributions discussed in Section~\ref{sec:pobosa} it holds for all $\ell$ that 
	\begin{align}
	\label{eq:clear}
	(\mX_n \mid |\mX_n| = \ell) \eqdist \mA_\ell.
	\end{align}
	Thus, setting $K_n = |X|_n$, it holds that $\mX_n \eqdist \mA_{K_n}$.
	By claim 2) of Lemma~\ref{le:pre11} we have $|K_n| = n + O_p(1)$, hence the remainder that gets attached to the root of $\mX_n$ to form the tree $\mS_n$ is stochastically bounded. This completes the proof 
	\end{proof}
	
	It remains to verify Lemma~\ref{le:pre11}.
		
	\begin{proof}[Proof of Lemma~\ref{le:pre11}]
	 We start with the first claim. By the discussion of Boltzmann samplers in Section~\ref{sec:pobosa} regarding the product and substitution operation, the probability generating function of $H_n$ is given by 
	\begin{align}
	\label{eq:expr1}
	\Ex{w^{H_n}} = \frac{[z^{n-1}]Z_{\Set_{\Omega}}(\tilde{\cA}_{\Omega^*}(\rho z), \tilde{\cA}_{\Omega^*}((\rho wz)^2), \tilde{\cA}_{\Omega^*}((\rho wz)^3), \ldots)}{[z^{n-1}]Z_{\Set_{\Omega}}(\tilde{\cA}_{\Omega^*}(\rho z), \tilde{\cA}_{\Omega^*}((\rho z)^2), \ldots)}.
	\end{align}
	Let us explain this argument in more detail. By the product rule it suffices, to study $(n-1)$-sized symmetries of $\Set_\Omega \circ \cA_{\Omega^*}$. The substitution rule tells us that a Boltzmann distributed symmetry of this composition with parameters $(\rho^i)_{i \ge 1}$  is obtained by first drawing a P\'olya--Boltzmann distributed $\Set_\Omega$-symmetry with parameters $(\tilde{\cA}_{\Omega^*}(\rho^i))_{i \ge 1}$, and then for each $j \ge 1$ and each $j$-cycle of the symmetry an unlabelled Boltzmann distributed symmetry of $\cA_{\Omega^*}$ with parameters $(\rho^{ij})_{i \ge 1}$, of which $i$ identical copies are attached to the $\Set_\Omega$-symmetry. Given a $k \in \Omega$ sized permutation $\nu$, the probability for the $\Set_\Omega$-symmetry to assume this permutation is given by
	\begin{align}
		\label{eq:ex1}
		\frac{\tilde{\cA}_{\Omega^*}(\rho)^{\nu_1} \cdots  \tilde{\cA}_{\Omega^*}(\rho^{k})^{\nu_{k}}}{k! \widetilde{\Set_{\Omega} \circ \cA_{\Omega^*}}(\rho) }
	\end{align}
	Conditioned on this event, the  probability generating function for the size of the resulting object is given by 
	\begin{align}
		\label{eq:ex2}
		\left( \frac{ \tilde{\cA}_{\Omega^*}(\rho z)}{\tilde{\cA}_{\Omega^*}(\rho)} \right)^{\nu_1}  \left( \frac{ \tilde{\cA}_{\Omega^*}((\rho z)^2)}{\tilde{\cA}_{\Omega^*}(\rho^2)} \right)^{\nu_2} \cdots \left( \frac{ \tilde{\cA}_{\Omega^*}( (\rho z)^{k} )}{\tilde{\cA}_{\Omega^*}(\rho^{k})} \right)^{\nu_{k}}.
	\end{align}
	The exponents in the arguments are due to the fact that we attach $i$ identical copies of each tree corresponding to an $i$-cycle. If we additionally want to keep track of the volume of the trees corresponding to cycles with length at least $2$, we may form the corresponding bivariate probability generating function where $w$ corresponds to this parameter and $z$ to the total size by
	\begin{align}
\label{eq:ex3}
\left( \frac{ \tilde{\cA}_{\Omega^*}(\rho z)}{\tilde{\cA}_{\Omega^*}(\rho)} \right)^{\nu_1}  \left( \frac{ \tilde{\cA}_{\Omega^*}((\rho wz)^2)}{\tilde{\cA}_{\Omega^*}(\rho^2)} \right)^{\nu_2} \cdots \left( \frac{ \tilde{\cA}_{\Omega^*}( (\rho wz)^{k} )}{\tilde{\cA}_{\Omega^*}(\rho^{k})} \right)^{\nu_{k}}.
\end{align}
	Multiplying \eqref{eq:ex1} and \eqref{eq:ex3} and summing over all outcomes that correspond to objects with size $n-1$ yields
	\begin{align}
		\label{ex:4}
		 \frac{[z^{n-1}]Z_{\Set_{\Omega}}(\tilde{\cA}_{\Omega^*}(\rho z), \tilde{\cA}_{\Omega^*}((\rho wz)^2), \tilde{\cA}_{\Omega^*}((\rho wz)^3), \ldots)}{\widetilde{\Set_{\Omega} \circ \cA_{\Omega^*}}(\rho) }.
	\end{align}
	Likewise multiplying \eqref{eq:ex1} with \eqref{eq:ex2} and summing up in the same way yields
	\begin{align}
		\label{ex:5}
		\frac{ [z^{n-1}]Z_{\Set_{\Omega}}(\tilde{\cA}_{\Omega^*}(\rho z), \tilde{\cA}_{\Omega^*}((\rho z)^2), \ldots)}{\widetilde{\Set_{\Omega} \circ \cA_{\Omega^*}}(\rho) }.
	\end{align}
	The quotient of \eqref{ex:4} and \eqref{ex:5} is the probability generating function for the random number $H_n$, and the expression obtained in this way agrees with Equation~\eqref{eq:expr1}.

	Having verified Equation~\eqref{eq:expr1} we proceed with the argument. Since $1 \in \Omega$ we may bound the denominator in~\eqref{eq:expr1} from below by 
	$[z^{n-1}]\tilde{\cA}_{\Omega^*}(\rho z)$, and by Proposition~\ref{SeFrpro:cnt} we have that \begin{align}
	\label{eq:asy}
	[z^{n-1}]\tilde{\cA}_{\Omega^*}(\rho z) \sim C n^{-3/2}
	\end{align} for some constant $C>0$ as $n \equiv 2 \mod \gcd(\Omega^*)$ tends to infinity. Moreover, for all $n$ the polynomial in the indeterminate $w$ in the numerator
	is dominated coefficient wise by the series \[Z_{\Set_{\Omega}}(\tilde{\cA}_{\Omega^*}(\rho), \tilde{\cA}_{\Omega^*}((\rho w)^2), \ldots)\] which by Proposition~\ref{SeFrpro:cnt} has radius of convergence strictly greater than $1$. In particular we have that \[\sum_{k \ge x} [w^k]Z_{\Set_{\Omega}}(\tilde{\cA}_{\Omega^*}(\rho), \tilde{\cA}_{\Omega^*}((\rho w)^2), \ldots) = O(\gamma^x)\] for some constant $0 < \gamma < 1$. Hence there is a constant $C'$ such that \[\Pr{H_n \ge x} \le C' n^{3/2} \gamma^x
	\] for all $n$ and $x$. By the discussion of Boltzmann samplers in Section~\ref{sec:pobosa} regarding the product and substitution operation, the probability generating function for the random number $f_n$ is given by
	\[
	\Ex{w^{f_n}} = \frac{[z^{n-1}]Z_{\Set_{\Omega}}(w \tilde{\cA}_{\Omega^*}(\rho z), \tilde{\cA}_{\Omega^*}((\rho z)^2), \ldots)}{[z^{n-1}]Z_{\Set_{\Omega}}(\tilde{\cA}_{\Omega^*}(\rho z), \tilde{\cA}_{\Omega^*}((\rho z)^2), \ldots)}.
	\]
	The corresponding bound for the event $f_n \ge x$ follows by the same arguments as for the parameter~$H_n$. This proves claim~1).
	
	We proceed with showing claim 2). If $\Omega = \ndN$, then we may apply Lemma~\ref{le:part} to obtain that the largest component in a random $(n-1)$-sized multiset of unlabelled $\cA_{\Omega^*}$-objects has size $n + O_p(1)$. By claim $1)$ it follows that with high probability $H_n \le \log^2 n$. Thus the largest component  must correspond to a fixpoint, verifying claim $2)$ for this special case. In order to treat the general case, it suffices by similar arguments to show that the largest component in a random $(n-1)$-sized  unlabelled $\Set_\Omega \circ \cA_{\Omega^*}$-object has size $n + O_p(1)$. However, we cannot apply Lemma~\ref{le:part} directly, and hence argue as follows. 
	

	We need to show that for any sequence $t_n \to \infty$ the probability for all components in the random $\Set_\Omega \circ \cA_{\Omega^*}$-object to have size at most $n - t_n$ tends to zero. Using analogous arguments as in the justification of Equation~\eqref{eq:expr1}, we may express this probability by the product of the normalizing factor
	\begin{align}
	\label{eq:yo}
	([z^{n-1}]Z_{\Set_{\Omega}}(\tilde{\cA}_{\Omega^*}(\rho z), \tilde{\cA}_{\Omega^*}((\rho z)^2), \ldots))^{-1}
	\end{align}
	with the expression
	\begin{align}
		\label{eq:big}
		\sum_\nu \sum_{(a_{ij})_{i,j}} \left[ z^{n-1}\prod_{i,j} x_{ij}^{a_{ij}} \right] \prod_{i,j} \tilde{\cA}_{\Omega^*}( (\rho z)^i x_{ij}).
	\end{align}
	Here the sum index $\nu$ ranges over all permutations of sets of the form $[k]$ for $k \in  \Omega$. The indices $(a_{ij})_{i,j}$ range over all families of numbers $a_{ij}$ with $1 \le i \le n-1$, $1 \le j \le \nu_i$, and such that $a_{ij} \le n - t_n$ for all $i,j$ and 
	\[
		\sum_{\substack{1 \le i \le n-1 \\ 1 \le j \le \nu_i}} ia_{ij} = n-1.
	\]
	The indices $i,j$ of the product range over all pairs of integers with $1 \le i \le n-1$ and $1 \le j \le \nu_i$.
	
	Applying a standard result for the singularity analysis of functions with multiple dominant singularities \cite[Thm. VI.5]{MR2483235} we obtain similarly as in Equation~\eqref{eq:identical} that the factor in Equation~\eqref{eq:yo} is asymptotically equivalent to $n^{3/2}$ times a constant. Thus, showing that the largest component in a random unlabelled $n-1$-sized $\Set_\Omega \circ \cA_{\Omega^*}$-object has size $n + O_p(1)$ is actually equivalent to showing that the expression in \eqref{eq:big} multiplied by $n^{3/2}$ tends to zero as $n \equiv 2 \mod \gcd(\Omega^*)$ becomes large. Consider the species $\bar{\cA}_{\Omega^*}$ where for each  $k \in \ndN_0$ we set $\bar{\cA}_{\Omega^*}[k] = \cA_{\Omega^*}[k - \ell]$  for the smallest integer $\ell \ge 0$ satisfying $k - \ell \in \Omega^*$. Hence $\cA_{\Omega^*}$ is a subspecies of $\bar{\cA}_{\Omega^*}$, and $ \tilde{\bar{\cA}}_{\Omega^*}(z)$ has the same radius of convergence as $\tilde{\cA}_{\Omega^*}(z)$. 
	
	We may apply Lemma~\ref{le:part} to the composition $\Set \circ \bar{\cA}_{\Omega^*}$, yielding that the expression obtained from \eqref{eq:big} by letting $\nu$ range over arbitrarily sized permutations and replacing $\tilde{\cA}_{\Omega^*}(\cdot)$ with $\tilde{\bar{\cA}}_{\Omega^*}(\cdot)$  belongs to the class $o(n^{-3/2})$ of sequences that still tend to zero when multiplied by $n^{3/2}$. But this expression is clearly an upper bound to the expression in \eqref{eq:big}, yielding that \eqref{eq:big} also belongs to $o(n^{-3/2})$. Hence the largest component in a random $(n-1)$-sized unlabelled $\Set_{\Omega} \circ \cA_{\Omega}$ object has size $n + O_p(1)$. This verifies claim 2).

	It remains to prove claim 3), i.e. we have to show that $\Ex{f_n} = O(1)$. If $\Omega \subset \ndN$ is bounded, then this is trivial. Otherwise it seems to require some work. We have that
	\[
	\Ex{f_n} = \frac{[z^{n-1}]\left(s_1 \frac{\partial Z_{\Set_\Omega}}{\partial s_1} \right)(\tilde{\cA}_{\Omega^*}(z),\tilde{\cA}_{\Omega^*}(z^2), \ldots)}{[z^{n-1}] Z_{\Set_\Omega}(\tilde{\cA}_{\Omega^*}(z),\tilde{\cA}_{\Omega^*}(z^2), \ldots)}.
	\]
	Since $1 \in \Omega$ the denominator is bounded from below by $[z^{n-1}]\tilde{\cA}_{\Omega^*}(z)$. By Proposition~\ref{SeFrpro:cnt} it follows that
	\[([z^{n-1}]\tilde{\cA}_{\Omega^*}(z))^{-1} = O(n^{3/2} \rho^n).\] The power series in $z$ in the numerator is bounded coefficient wise by 
	\[
	\left(s_1 \frac{\partial Z_{\Set}}{\partial s_1} \right)(\tilde{\cA}_{\Omega^*}(z),\tilde{\cA}_{\Omega^*}(z^2), \ldots) = \tilde{\cA}_{\Omega^*}(z) \exp \left(\sum_{i=1}^\infty \tilde{\cA}_{\Omega^*}(z^i)/i\right) = h(\tilde{\cA}_{\Omega^*}(z))g(z)
	\]
	with \[h(w) = w \exp(w)\] being analytic on $\ndC$ and \[g(w) = \exp\left(\sum_{i \ge 2} \tilde{\cA}_{\Omega^*}(z^i)/i\right)\] having radius of convergence strictly larger than $\rho$ since $\rho < 1$. By a singularity analysis using results from \cite{MR2240769} and \cite[Thm. VI.5]{MR2483235} it follows that \[[z^{n-1}]h(\tilde{\cA}_{\Omega^*}(z))g(z) = O(n^{-3/2} \rho^{-n}).\] The detailed arguments are identical as in the proof of Proposition~\ref{SeFrpro:together}. This concludes the proof.
\end{proof}

\subsubsection{Symmetrically cycle pointed trees whose cycle center is a vertex}
Recall that
\[
	\cV = (\Set_{\Omega}^\circledast \circledcirc \cA_{\Omega^*}) \star \cX.
\]
Let $(\mV_n, \tau_n, \sigma, v_n)$ be a rooted c-symmetry drawn uniformly at random from the set $\RSym(\cV)[n]$. In particular, $\mV_n$ is distributed like the uniformly at random chosen unlabelled $\cV$-object with size~$n$. Let $\pi_n$ denote the corresponding partition.  By the discussion in Section~\ref{SeFrsec:oponsp}, $\sigma$ induces an automorphism \[\bar{\sigma}: \pi_n \to \pi_n\] of the $\Set_\Omega$-object. Moreover, let $F_n \subset \pi_n$ denote the fixpoints of $\bar{\sigma}$, $f_n = |F_n|$ their number and for each fixpoint $Q \in F_n$ let $(\mA_Q, \sigma_Q)$ denote the corresponding symmetry from $\Sym(\cA_{\Omega^*})[Q]$. Let $H_n$ denote the total size of the trees dangling from cycles with length at least $2$. We are going to make the following observations.

\begin{lemma}
	\label{le:reduct}
	The following statements hold.
	\begin{enumerate}[1)]
		\item There are constants $C_1>0$ and $0<\gamma<1$ such that for all $n$ and $x \ge 0$ we have that \[\Pr{H_n \ge x} \le C_1 n^{3/2} \gamma^x\] and \[\Pr{f_n \ge x} \le C_1 n^{3/2} \gamma^x.\]
		\item The maximum size  of the trees corresponding to the fixpoints of $\bar{\sigma}$ satisfies \[  
		\max_{Q \in F_n}|\mA_Q| = n + O_p(1).
		\]
		\item There is a constant $C_2>0$ such that \[\Ex{f_n} \le C_2\] for all $n$.
	\end{enumerate}
\end{lemma}
	From these claims we may deduce  Lemma~\ref{SeFrle:vobj} in an entirely analogous manner as we deduced Lemma~\ref{SeFrle:sobj} from Lemma~\ref{le:pre11}. We leave the details to the reader.
It remains to verify Lemma~\ref{le:reduct}.
\begin{proof}[Proof of Lemma~\ref{le:reduct}]
We start with claim 1). Using the Boltzmann-sampling methods from Section~\ref{sec:pobocyc}
, we obtain that the  probability generating function of $H_n$ is given by
	\begin{align}
	\label{eq:indetail}
	\Ex{w^{H_n}} = \frac{[z^{n-1}]\bar{Z}_{\Set^\circledast_{\Omega}}(\tilde{\cA}_{\Omega^*}(\rho z), \tilde{\cA}^\circ_{\Omega^*}(\rho z); \tilde{\cA}_{\Omega^*}((\rho wz)^2),\tilde{\cA}^\circ_{\Omega^*}((\rho wz)^2); \ldots)}{[z^{n-1}]\bar{Z}_{\Set^\circledast_{\Omega}}(\tilde{\cA}_{\Omega^*}(\rho z), \tilde{\cA}^\circ_{\Omega^*}(\rho z); \tilde{\cA}_{\Omega^*}((\rho z)^2), \tilde{\cA}^\circ_{\Omega^*}((\rho z)^2); \ldots)}.
	\end{align}
	A detailed justification of this fact goes as follows. By the product rule in Section~\ref{sec:pobocyc} it suffices to consider $(n-1)$-sized rooted symmetries of $\Set_{\Omega}^{\circledast} \circ \cA_{\Omega^*}$.  The composition rule states that to sample such a symmetry according  according to the Boltzmann model with parameters $(\rho^i, \rho^i)_{i \ge 1}$, we may start with a P\'olya--Boltzmann distributed rooted symmetry of $\Set_\Omega^{\circledcirc}$ with parameters $(\tilde{\cA}_{\Omega^*}(\rho^i), \tilde{\cA}^\circ_{\Omega^*}(\rho^i))_{i \ge 1}$. Then, for each $j \ge 1$ and each unmarked $j$-cycle  a  symmetry of $\cA_{\Omega^*}$ is sampled according to a P\'olya--Boltzmann distribution with parameters $(\rho^{ij})_i$, and for the marked cycle we let $s$ denote its length and draw a rooted symmetry of $\cA_{\Omega^*}$  according to a P\'olya--Boltzmann distribution with parameters $(\rho^{si}, \rho^{si})_{i \ge 1}$. Given a $k \in \Omega$ sized permutation $\nu$ with a marked cycle having length $\ell \ge 2$ and a distinguished atom of this cycle, the probability for the rooted symmetry of $\Set_{\Omega}^\circledcirc$ to assume this value is given by
	\begin{align}
		\label{eq:t1}
		\frac{ \tilde{\cA}_{\Omega^*}^\circ(\rho^\ell) \tilde{\cA}_{\Omega^*}(\rho^\ell)^{\nu_\ell -1} }{k! \bar{Z}_{\Set_\Omega^\circledcirc}(\tilde{\cA}_{\Omega^*}(\rho), \tilde{\cA}^\circ_{\Omega^*}(\rho); \tilde{\cA}_{\Omega^*}(\rho^2), \tilde{\cA}^\circ_{\Omega^*}(\rho^2); \ldots)} \prod_{\substack{1 \le i \le k \\ i \ne \ell}} \tilde{\cA}_{\Omega^*}(\rho^i)^{\nu_i}.
	\end{align}
	Conditioned on this event, the probability generating function for the size of the resulting object is given by
	\begin{align}
		\label{eq:t2}
		\frac{ \tilde{\cA}^\circ_{\Omega^*}( (\rho z)^\ell)}{\tilde{\cA}^\circ_{\Omega^*}(\rho^\ell)} \prod_{\substack{1 \le i \le k \\ i \ne \ell}}  \left( \frac{ \tilde{\cA}_{\Omega^*}( (\rho z)^i)}{\tilde{\cA}_{\Omega^*}(\rho^i)} \right)^{\nu_i - \one_{i = \ell}}.
	\end{align}
	The exponents $(\rho z)^i$ are due to the fact that for each object corresponding to an $i$-cycle we attach $i$~identical copies, and likewise for the marked cycle. In order to keep track of the volume of the trees corresponding to cycles with length at least $2$ we may form the bivariate probability generating function where the variable $w$ corresponds to this parameter and $z$ to the total size, given by
	\begin{align}
		\label{eq:t3}
	\frac{ \tilde{\cA}^\circ_{\Omega^*}( (\rho wz)^\ell)}{\tilde{\cA}^\circ_{\Omega^*}(\rho^\ell)} \left( \frac{ \tilde{\cA}_{\Omega^*}( \rho z)}{\tilde{\cA}_{\Omega^*}(\rho)} \right)^{\nu_1} \prod_{\substack{2 \le i \le k \\ i \ne \ell}}  \left( \frac{ \tilde{\cA}_{\Omega^*}( (\rho wz)^i)}{\tilde{\cA}_{\Omega^*}(\rho^i)} \right)^{\nu_i - \one_{i = \ell}}.
	\end{align}
	Multiplying~\eqref{eq:t1} with \eqref{eq:t3} and summing over all outcomes with total size $n-1$ yields
	\begin{align}
		\label{eq:s1}
		 \frac{[z^{n-1}]\bar{Z}_{\Set^\circledast_{\Omega}}(\tilde{\cA}_{\Omega^*}(\rho z), \tilde{\cA}^\circ_{\Omega^*}(\rho z); \tilde{\cA}_{\Omega^*}((\rho wz)^2),\tilde{\cA}^\circ_{\Omega^*}((\rho wz)^2); \ldots)}{\bar{Z}_{\Set_\Omega^\circledcirc}(\tilde{\cA}_{\Omega^*}(\rho), \tilde{\cA}^\circ_{\Omega^*}(\rho); \tilde{\cA}_{\Omega^*}(\rho^2), \tilde{\cA}^\circ_{\Omega^*}(\rho^2); \ldots)}.
	\end{align}
	Multiplying~\eqref{eq:t1} with \eqref{eq:t2} and summing over all outcomes with total size $n-1$ yields
	\begin{align}
		\label{eq:s2}
	\frac{[z^{n-1}]\bar{Z}_{\Set^\circledast_{\Omega}}(\tilde{\cA}_{\Omega^*}(\rho z), \tilde{\cA}^\circ_{\Omega^*}(\rho z); \tilde{\cA}_{\Omega^*}((\rho z)^2),\tilde{\cA}^\circ_{\Omega^*}((\rho z)^2); \ldots)}{\bar{Z}_{\Set_\Omega^\circledcirc}(\tilde{\cA}_{\Omega^*}(\rho), \tilde{\cA}^\circ_{\Omega^*}(\rho); \tilde{\cA}_{\Omega^*}(\rho^2), \tilde{\cA}^\circ_{\Omega^*}(\rho^2); \ldots)}.
	\end{align}
	The quotient of \eqref{eq:s1} and \eqref{eq:s2} is the probability generating function for the parameter $H_n$, and the expression obtained in this way agrees with Equation~\eqref{eq:indetail}.

	Having verified Equation~\eqref{eq:indetail}, we proceed with the argument.
	Since $1 \in \Omega$ and there is a number $k \ge 3$ with $k \in \Omega$ it follows that the denominator in \eqref{eq:indetail} is bounded from below by \[[z^{n-1}] z^{k-1} \tilde{\cA}_{\Omega^*}(\rho z) = [z^{n-k}] \tilde{\cA}_{\Omega^*}(\rho z).\]
	We have that \[n - k \equiv 1 \mod \gcd(\Omega^*)\] and thus, by Proposition~\ref{SeFrpro:cnt}, we have that \[[z^{n-k}] \tilde{\cA}_{\Omega^*}(\rho z) \sim C n^{-3/2}\] as $n \equiv 2 \mod \gcd(\Omega^*)$ tends to infinity. The polynomial in the numerator with indeterminate $w$ is bounded coefficient wise by the series \[\bar{Z}_{\Set^\circledast_{\Omega}}(\tilde{\cA}_{\Omega^*}(\rho ), \tilde{\cA}^\circ_{\Omega^*}(\rho ); \tilde{\cA}_{\Omega^*}((\rho w)^2),\tilde{\cA}^\circ_{\Omega^*}((\rho w)^2); \ldots)\] which does not depend on $n$ and, by Proposition~\ref{SeFrpro:ser3}, has radius of convergence strictly larger than $1$.
	It follows that there is a constant $C'$ such that \[\Pr{H_n \ge x} \le C' n^{3/2} \gamma^x\] for all $n$ and $x$. By a similar argument as for Equation~\eqref{eq:indetail} the probability generating function for the random number number $f_n$ is given by
	\[
	\Ex{w^{f_n}} = \frac{[z^{n-1}]\bar{Z}_{\Set^\circledast_{\Omega}}(w \tilde{\cA}_{\Omega^*}(\rho z), w \tilde{\cA}^\circ_{\Omega^*}(\rho z); \tilde{\cA}_{\Omega^*}((\rho z)^2),\tilde{\cA}^\circ_{\Omega^*}((\rho z)^2); \ldots)}{[z^{n-1}]\bar{Z}_{\Set^\circledast_{\Omega}}(\tilde{\cA}_{\Omega^*}(\rho z), \tilde{\cA}^\circ_{\Omega^*}(\rho z); \tilde{\cA}_{\Omega^*}((\rho z)^2), \tilde{\cA}^\circ_{\Omega^*}((\rho z)^2); \ldots)}.
	\]
	The corresponding bound for the event $f_n \ge x$ follows by the same arguments as for $H_n$. This proves claim~1). Claims 2) and 3) follow by analogous arguments as in the proofs of claims 2) and 3) in Lemma~\ref{le:pre11}. 
\end{proof}

\bibliographystyle{abbrv}
\bibliography{biblio}

\end{document}